\documentclass[reqno,11pt]{amsart}

\usepackage[T1]{fontenc} 

\usepackage{amssymb}
\usepackage{amsmath}
\usepackage{amsfonts}
\usepackage{amsthm}
\usepackage{mathtools}
\usepackage{mathabx}
\usepackage{fullpage}
\usepackage{comment}
\usepackage[dvipsnames]{xcolor}
\usepackage{graphicx}
\usepackage{float}
\usepackage{enumerate}

\usepackage{url} 
\usepackage{hyperref}

\usepackage{pgfplots}
\pgfplotsset{compat=1.6}

\usepackage{todonotes}
\theoremstyle{plain}%
\newtheorem{theorem}{Theorem}[section]
\newtheorem{lemma}[theorem]{Lemma}
\newtheorem{proposition}[theorem]{Proposition}

\newtheorem*{conjecture*}{Conjecture}

\theoremstyle{definition}
\newtheorem{definition}[theorem]{Definition}
\newtheorem{example}[theorem]{Example}

\theoremstyle{remark}
\newtheorem{remark}[theorem]{Remark}

\usepackage{soul}

\let \le \leqslant
 \let \leq \leqslant
 \let \geq \geqslant
 \let \ge \geqslant

\usepackage{fancyhdr}
\usepackage{graphics}
\usepackage{graphicx}

\title{A formula for hidden regular variation behavior for symmetric
stable distributions 
}

\author{Malin P. Forsstr\"om}
\address[Malin P. Forsstr\"om]{KTH Royal Institute of Technology, Stockholm, Sweden.}
\email{malinpf@kth.se}

\author{Jeffrey E. Steif}
\address[Jeffrey E. Steif]{Chalmers University of Technology and Gothenburg University, Gothenburg, 
Sweden}
\email{steif@chalmers.se}

\usepackage{url}
\usepackage{float}
\usepackage[dvipsnames]{xcolor}
\usepackage{enumerate}
\usepackage{tikz}
\usepackage{caption}
\usepackage{cases}

\usepackage{subfigure}

\usepackage{pgfplots}
\usepackage{tikz}
\usetikzlibrary{quotes,angles,calc}
\usepackage{tikz-3dplot}

\usepgfplotslibrary{fillbetween}

\newcommand{\bftheta}{\mathbf{\theta}}
\newcommand{\E}{\mathbb{E}}

\DeclareMathOperator{\unif}{unif}

\DeclareMathOperator{\support}{supp}

\keywords{Hidden regular variation, multivariate stable distributions.}
\subjclass[2010]{60E07, 60G70.}
\begin{document}

\maketitle

\begin{abstract}
We develop a formula for the power-law decay of various sets for symmetric stable random vectors
in terms of how many vectors from the support of the corresponding spectral measure are needed to enter the set.
One sees different decay rates in ``different directions'', illustrating the phenomenon of 
\emph{hidden regular variation}. We give several examples and obtain quite varied behavior, including sets 
which do not have exact power-law decay.

\end{abstract}

\section{Main result and remarks}\label{section: main} 
  
Many distributions have tails that exhibit regular variation 
(see \cite{BGT} and \cite{MW2019}) which means that they
behave like a power-law times a slowly varying function. Examples are
one-dimensional stable distributions where the slowly varying function is just constant.
For stable random vectors, one also has this but in addition,
one can have more interesting behavior, so-called
{\em hidden regular variation} (see \cite{dmr2013}, \cite{sr2002}, \cite{sr2007}), 
meaning that one has different power-law decay in different directions. Ideally, one would like to
capture the correct decay rate in each such direction.
Our main result, Theorem~\ref{theorem: later tails}, describes such behavior for symmetric
stable distributions. Needed definitions and background will be given in 
Section~\ref{section: background}.

Let \( \alpha \in (0,2) \) and \( X \) be an \(n \)-dimensional 
symmetric \( \alpha \)-stable random vector with spectral measure \( \Lambda \) (see~\eqref{eq: multivariate stable cf}). Then \( \Lambda \) is a bounded measure on the unit sphere \( \mathbb{S}^{n-1} \) in \( \mathbb{R}^n \).
Let \( E \subseteq \mathbb{R}^n \) be a Borel set with \( \mathbf{0} \not \in \bar E \),
where \( \bar E \) and \( E^o \) denote the closure and interior of $E$ respectively.
With $d$ being the Euclidean distance, define the \( \delta \)-neighborhood of \(E \) by
$$ 
E_{\delta,+} \coloneqq \{ \mathbf{x} \in \mathbb{R}^n \colon d(\mathbf{x},E) < \delta \}.
$$
For any integer $k\ge 1$ and any $E$ as above, letting $C_\alpha$ be a constant
defined in Section~\ref{section: background} and \( \Lambda^k \coloneqq \Lambda \times \cdots \times\Lambda \) (\( k \) times), 
 define 
 \begin{equation}\label{eq: the limiting expression 0}
\begin{split}
L(E,k,\alpha)\coloneqq  \frac{C_\alpha^k  }{k!}
\int_0^\infty
 \cdots
\int_{0}^\infty
& \Lambda^k \Biggl( \biggl\{ \mathbf{x}_1, \ldots, \mathbf{x}_k \in \mathbb{S}^{n-1} \colon   \sum_{i=1}^k s_i\mathbf{x}_i \in E \biggr\} \Biggr) 
\cdot    \prod_{i=1}^k \alpha  s_i^{-(1+\alpha)}\, ds_i.
\end{split}
\end{equation}

\begin{theorem} \label{theorem: later tails}
For any \( \alpha \), \( X \), $E$ and $k$ as above,
\begin{equation}\label{eq:theorem1}
\begin{split}
&L(E^o,k,\alpha)\le \liminf_{h\to \infty} h^{k \alpha} P(X \in h E )
\\&\qquad \le \limsup_{h\to \infty}  h^{k \alpha} P(X \in h E ) \le \lim_{\delta \to 0} L(E_{\delta,+},k,\alpha).
\end{split}
\end{equation}

\end{theorem}

\begin{remark}\label{remark: 1}
\mbox{}
\begin{enumerate}[(i)] 

\item 
Clearly \(L(E^o,k,\alpha)\le L(\bar E,k,\alpha)\le  \lim_{\delta \to 0} L(E_{\delta,+},k,\alpha)\). Also, by the Lebesgue dominated convergence theorem, 
\(L(\bar E,k,\alpha)= \lim_{\delta \to 0}L(E_{\delta,+},k,\alpha)\) provided that
\(L(E_{\delta,+},k,\alpha)\) is finite for some $\delta>0$.

\item  Since $0\not\in \bar E$, for small \( \delta>0 \) and $s_1>0$ we have that $s_1\mathbb{S}^{n-1}\cap E_{\delta, +}=\emptyset$. This implies in particular that  the integrand in the definition of \( L(E_{\delta,+},1,\delta)\) is equal to zero for small \( \delta \) and  $s_1$, removing the singularity at $s_1=0$, and hence $\lim_{\delta \to 0} L(E_{\delta,+},1,\alpha)$ is always finite.
\item We let \( \support (\Lambda) \) be the (topological) support of \( \Lambda\) and, for \( k \in \{  1,2, \ldots, n \} \), let
\begin{equation*}
\begin{split}
S_\Lambda(k) \coloneqq \Bigl\{ \mathbf{x} \in \mathbb{R}^n \colon \mathbf{x} = \sum_{i = 1}^k s_i \mathbf{x}_i &\text{ for some } s_1,\ldots, s_k \in \mathbb{R} \text{ and } \\[-3ex]&\qquad\qquad \mathbf{x}_1, \ldots, \mathbf{x}_k \in \support(\Lambda) \Bigr\}.
\end{split}
\end{equation*} 
Then \( \lim_{\delta \to 0} L(E_{\delta,+},k,\alpha) = 0\) whenever \( E_{\varepsilon,+} \cap S_\Lambda(k) = \emptyset\) for sufficiently small \( \varepsilon > 0 \). 
On the other hand, if \( E^o \cap S_{\Lambda}(k) \not = 0 \), then \( L(E^o,k,\alpha) > 0 \). 
Finally, assume that \( E \subseteq \overline{E^o}\), \( E^o \cap S_{\Lambda}(k) \not = 0 \) and there is \( \varepsilon>0 \) such that  whenever \( \sum_{i=1}^k s_i \mathbf{x}_i \in E^o \) for some \( \mathbf{x}_1,\ldots, \mathbf{x}_k \in \support \Lambda \), then  \( |s_1|, \ldots, |s_k| > \varepsilon \). This implies that if \( \delta \in ( 0,\varepsilon )\), then \( L(E_{\delta,+}, k, \alpha) < \infty\), 
	and hence by (i), \( L(\bar E, k, \alpha) = \lim_{\delta \to 0} L(E_{\delta,+}, k, \alpha)\). Furthermore, if in addition we assume that the boundary of \( E \) has zero Lebesgue measure, then \( L(E^o,k,\alpha) = L(\bar E, k,\alpha) \) and hence by Theorem~\ref{theorem: later tails}, \( \lim_{h \to \infty} h^{k \alpha} P(X \in hE) = L(E,k, \alpha) \in (0,\infty).\) This illustrates what would typically hold in most generic or ''nice'' situations.
\end{enumerate}
\end{remark}

\begin{remark}
	Taking \( E \) to be the set \( \{ \mathbf{x} \in \mathbb{R}^n \colon \min x_i > 1 \} \) and 
\( k = 1 \), one obtains Theorem~4.4.1 in~\cite{st1994} in the symmetric 
case (see equation (4.4.2) in~\cite{st1994}). 
\end{remark}

\begin{remark}
	If we for \( \mathbf{x} \in \mathbb{R}^n \) let \( \| \mathbf{x} \|_2 \coloneqq (x_1^2 + \ldots + x_n^2)^{1/2} \) and  define
\[
 E \coloneqq \rm{Cone}(A) \coloneqq \bigl\{ \mathbf{x} \in \mathbb{R}^n \colon \| \mathbf{x} \|_2   > 1 \text{ and } \mathbf{x}/ \|\mathbf{x}\|_2 \in A \bigr\}
 \]
 for some \( A \subseteq \mathbb{S}^{n-1} \) with \(\Lambda(\partial A) =0,\)  
and then apply Theorem~\ref{theorem: later tails} 
to both $E$ and to the complement of the unit ball
with \( k = 1 \), we recover Corollary 6.20 in~\cite{ag1980} in the symmetric case, stating that
\begin{equation}\label{eq: cone result}
\lim_{h \to \infty} \frac{P\bigl(X \in \rm{Cone}(A),\, \| X \|_2 > h\bigr)}{P\bigl(\| X \|_2 > h\bigr)} = \frac{\Lambda(A)}{\Lambda\bigl(\mathbb{S}^{n-1}\bigr)}.
\end{equation}

\end{remark}

\begin{remark}
	Our motivation for looking at Theorem~\ref{theorem: later tails} and its consequences (see 
Section~\ref{section: examples}) was to understand which
threshold stable vectors can be obtained as divide and color processes in the sense of 
\cite{st2017}. These applications, as well as a study of which 
threshold Gaussian vectors can be obtained as divide and color processes, is carried out in 
\cite{PS.threshold}. 
\end{remark}

\begin{remark} 
 
One might  guess that Theorem~\ref{theorem: later tails} would generalize to so-called regularly varying random vectors (see e.g.\ Proposition 2.2.20 on p.\ 57 in~\cite{MW2019}). This 
is however not the case. To see this, let \( \alpha \in (0,2) \) and let \( X_1 \) be an  
\( \alpha \)-stable random vector in \( \mathbb{R}^2 \) whose spectral measure \( \Lambda_1 \) has 
mass \( 1/4 \) at \( \pm (1,0) \) and \( \pm (0,1) \). Further, let \( \alpha' \in (\alpha, 2\alpha) \cap (0,2)\) 
and let \( X_2 \) be an \( \alpha' \)-stable random vector in \( \mathbb{R}^2 \) with uniform 
spectral measure independent of \( X_1 \).
Then  \( X_1 \) and  \( X_1+X_2 \) are both regularly varying with the same index and same limiting measure. However, if we let 
\( E \coloneqq \rm{Cone}(\pi/8, 3\pi/8)\),
then using Theorem~\ref{theorem: later tails} one easily obtains \( P(X_1 \in hE) \asymp h^{-2\alpha} \)
while, letting \( E'\coloneqq \rm{Cone}((\pi/8)+.001, (3\pi/8)-.001)\), we have for $h$ large
\[
2P\bigl(X_1 + X_2 \in hE\bigr) \geq  P\bigl(X_2 \in hE'\bigr) \asymp  h^{-\alpha'} \gg h^{-2\alpha}.
\]
(As usual, \( \asymp \) denotes two quantities whose ratio is bounded away from zero and infinity.)
\end{remark}

\section{Background}\label{section: background} 

\subsection{Stable random vectors}

In this section we now give some relevant definitions. These will be very brief as we assume the reader is 
familiar with the basics of stable vectors. For a more thorough introduction to stable random vectors, we refer the reader to~\cite{st1994}. 

\begin{definition}
A  random vector $X \coloneqq (X_i)_{1\le i\le n}$ in $\mathbb{R}^n$ has a symmetric 
\emph{stable} distribution if $X$ is symmetric (invariant under $\mathbf{x}\mapsto -\mathbf{x}$) and
if for all $k\geq 1$, there exists $a_k>0$ so that if 
$X^1$, $\ldots , $ $X^k$ are $k$ i.i.d.\ copies of $X$, then
$$
\sum_{1\le i\le k} X^i \overset{\mathcal D}{=} a_k X.
$$
\end{definition}
\noindent It is well known that for any symmetric stable vector $X$ there exists $\alpha\in (0,2]$, called the stability index,  so that 
for all $k\ge 1$, $a_k=k^{1/\alpha}$. The stability index $\alpha =2$ corresponds to  
Gaussian random vectors. If \( n = 1 \), then  besides $\alpha$, there is only one 
parameter, the scale parameter $\sigma$, and in this case
the characteristic function $\phi_X (\theta)$ is given by
$$
\phi_X(\theta)=e^{-\sigma^\alpha |\theta|^\alpha},\qquad  \theta \in \mathbb{R}.
$$
(When $\alpha=2$, $\sigma$ corresponds to the standard deviation divided by $\sqrt{2}$, 
an irrelevant scaling.) When $\sigma=1$, we denote this distribution by $S_{\alpha}$.
For stable vectors, the picture is somewhat more complicated.
A random vector \( X \) in $\mathbb{R}^n$ has a symmetric stable distribution with stability 
exponent $\alpha$ if and only if its characteristic function $\phi_X(\bftheta)$ has the form 
\begin{equation}\label{eq: multivariate stable cf}
\phi_X(\bftheta)=\exp \Bigl( -\int_{\mathbb{S}^{n-1}} |\bftheta\cdot \mathbf{x}|^\alpha \, d\Lambda(\mathbf{x}) \Bigr), \qquad  \theta \in \mathbb{R}^n
\end{equation}
for some finite measure $\Lambda$ on the unit sphere $\mathbb{S}^{n-1}$ 
which is invariant under
$\mathbf{x}\mapsto -\mathbf{x}$. $\Lambda$ is called the {\it spectral measure} 
of \( X \). If~\eqref{eq: multivariate stable cf} holds for some \( \alpha \) and \( \Lambda \), we write \( X \sim S_\alpha(\Lambda) \).  For $\alpha\in (0,2)$ fixed, different $\Lambda$'s yield different distributions. 
This is not true for $\alpha=2$.

When \( S_1 \), \( S_2 \), \ldots , \( S_m \) are i.i.d.\ random variables with distribution $S_\alpha$, 
$S \coloneqq (S_1,\dots,S_m)$,  and $A$ is an $n\times m$ matrix, then the vector $X \coloneqq (X_1,\ldots,X_n)$ defined by
$$
X \coloneqq  A S
$$
is a symmetric \( \alpha \)-stable random vector. To describe the spectral measure of \( X \), consider the columns 
of $A$ as elements of \(\mathbb{R}^n \), denoted by ${\hat {\mathbf{y}}}_1,\ldots,{\hat {\mathbf{y}}}_m$.
Then $\Lambda$ is obtained by placing, for each $i\in [m] \coloneqq \{ 1,2, \ldots, m\}$, a mass of weight
$ {\| {\hat {\mathbf{y}}}_i\|_2^\alpha}/{2}$ at $\pm {\hat {\mathbf{y}}}_i/\| {\hat {\mathbf{y}}}_i\|_2$. See p.\ 69 in \cite{st1994}.

Finally, we need the following facts. If \( X \sim S_\alpha \), then 
\begin{equation}\label{eq: basictail}
P (X \geq h ) \sim  \frac{C_\alpha  \, h^{-\alpha}}{2}  \qquad \text{as } h\to\infty
\end{equation}
where there is an exact formula for $C_\alpha$; see e.g.\ page 17 in \cite{st1994}.
The exact formula for this constant will not be relevant to us and so we will 
express quantities in terms of $C_\alpha$. Moreover, if we let \( f \) 
denote the probability density function of $X$, then
\begin{equation}\label{eq: densitytail}
f(h) \sim  \frac{C_\alpha\, \alpha h^{-(1+\alpha)}}{2} \qquad \text{as } h\to\infty;
\end{equation}
see~\cite{n1999}. Also, \( f (x) \) is decreasing in \( x \)  for \( x > 0 \);
see Theorem 2.7.4 on page 128 in \cite{z1986}.

\subsection{Related work}

When the spectral measure \( \Lambda \) of \( X \) is finitely supported, some asymptotic behavior of the corresponding probability 
density function \( f(x) \) in different directions is obtained in~\cite{h2003}. 
However, since the convergence in this case 
is not known to be uniform, this result cannot be used to get a version of Theorem~\ref{theorem: later tails}  for finitely supported \( \Lambda \). 
We also mention that results in~\cite{w2007} can be used to find the correct normalizing function above for many sets, but that these results cannot be used to find an expression for the limit as given by Theorem~\ref{theorem: later tails}, as only upper and lower bounds are given.
In both~\cite{h2003}~and~\cite{w2007}, the proofs are analytical, while our proofs are more probabilistic.

\section{Proof of Theorem~\ref{theorem: later tails}} \label{section: proof} 
 
The proof of Theorem~\ref{theorem: later tails} is somewhat simpler in the case when the spectral measure
is finitely supported in addition to being symmetric. We therefore first give a proof in this simpler setting, which is also sufficient for the examples covered in Section~\ref{section: examples}.

\begin{proof}[Proof of Theorem~\ref{theorem: later tails} for symmetric and finitely supported spectral measures] 
Suppose that \( \Lambda \) is symmetric and has support in 
\( \pm \mathbf{y}_1, \ldots, \pm \mathbf{y}_m \in \mathbb{S}^{n-1} \). For 
\( i = 1,2,\ldots m \),  let 
\({\hat {\mathbf{y}}}_i \coloneqq (2 \Lambda(\mathbf{y}_i))^{1/\alpha} \mathbf{y_i} \) and
let \( S_1,S_2, \ldots, S_m \sim S_\alpha \) be i.i.d. Then we have (see 
Section~\ref{section: background}) that
\[
X  = (X_1,X_2, \ldots, X_n) \overset{\mathcal{D}}{=} {\hat {\mathbf{y}}}_1 S_1 + \ldots + {\hat {\mathbf{y}}}_m S_m.
\]

The rest of the proof will be divided into two steps. In the first step, we give a proof under the additional assumption that,  for some positive integer \( k \),
\begin{equation}\label{eq: the k point condition}
\forall ( s_1,\ldots, s_m) \in \mathbb{R}^m  \colon \sum_{i=1}^m s_i {\hat {\mathbf{y}}}_i \in \bar E \Rightarrow | \{ i \in [m] \colon s_i \not = 0 \}| \geq k.
\end{equation}
In the second step, we show that this additional assumption can be removed.

\paragraph{Step 1.}

Assume that~\eqref{eq: the k point condition} holds. Given this assumption, we make the following observations.
\begin{description}
\item[(O1)] The assumption on \( \bar E \) in~\eqref{eq: the k point condition} implies that there is \( \varepsilon_0 > 0 \) 
such that if \( s_1, s_2, \ldots, s_m \) are such that \( \sum_{i = 1}^m s_i {\hat {\mathbf{y}}}_i \in \bar E  \), then there is a set \( J \subseteq [m] \), with \( | J | \geq k \), 
such that \( |s_i |> \varepsilon_0 \) for all \( i \in J \).
\item[(O2)] It follows from the previous observation and~\eqref{eq: basictail} that
\begin{equation}\label{eq: order}
P \bigl(X \in h \bar E \bigr) = O \bigl(h^{-k \alpha} \bigr) .
\end{equation}
\item[(O3)] For any \( \varepsilon' > 0 \),  \[ P \Bigl[|\{ i \in [m] \colon |S_i|> \varepsilon' h|>k \Bigr] = o(h^{-\alpha k} ) .\] 
\end{description}

For each \( \delta > 0 \), recall
\[
 E_{\delta,+} = \bigl\{ \mathbf{x} \in \mathbb{R}^n \colon d(\mathbf{x},E) < \delta \bigr\} \]
 and define
 \[
 E_{\delta,-} \coloneqq \bigl\{ \mathbf{x} \in E \colon d(x,\partial E)> \delta \bigr\}.
 \]
Using the observations above, it follows that for any \( \varepsilon' \in (0,\varepsilon_0 ) \)
\begin{align*}
 & P\bigl( X \in hE  \bigr) 
  =
 \sum_{J \subseteq [m] \colon |J| = k} P\Bigl[ X \in hE \text{ and }   \forall i \in [m] \backslash J \colon |S_i| \leq \varepsilon' h \Bigr]  + o\bigl(h^{-k \alpha}\bigr).
 \end{align*}
Fix $\delta >0$ arbitrarily and set  
\( \varepsilon' = \varepsilon_0 \land \bigl( \delta/((m-k) \sup_{i \in [m]} \| {\hat { \mathbf{y}}}_i \|_2 ) \bigr) \). 
Note that for each set \( J \), the event in question implies that
\(  |S_i|\leq \delta h/\bigl((m-k)  \sup_{i \in [m]} \| {\hat { \mathbf{y}}}_i \|_2 ) \bigr) \) for all \( i \in [m] \backslash J \), 
which in turn implies that \( \| \sum_{i \in [m] \backslash J} S_i \hat y_i \|_2 \leq \delta h \). 
Hence the previous equation can be bounded from below by
 \begin{align*}
& \sum_{J \subseteq [m] \colon |J| = k} P\Bigl[\, \sum_{i \in J} S_i {\hat {\mathbf{y}}}_i \in hE_{\delta,-}, \,  \text{ and }   \forall i \in [m] \backslash J \colon |S_i| \leq \varepsilon' h \Bigr] + o(h^{-k \alpha})
\\&\qquad \geq
\sum_{J \subseteq [m] \colon |J| = k} P\Bigl[ \, \sum_{i \in J} S_i {\hat {\mathbf{y}}}_i \in hE_{\delta,-}\Bigr] + o(h^{-k \alpha}).
\end{align*} 
Let \( f \) denote the common probability density function of \( S_1 \), \(S_2 \), \( \ldots \), \(S_m \). 
By (O1), we have that for a fixed set \( J \) of size $k$,
\begin{align*}
&   P\Bigl[ \,\sum_{i \in J} S_i {\hat {\mathbf{y}}}_i \in hE_{\delta,-} \Bigr]
=  \int_{ s_1, \ldots, s_k \in \mathbb{R} \colon \atop |s_1|,\ldots, |s_k | > \varepsilon_0 h}
I\Bigl( \sum_{i \in J} s_i {\hat {\mathbf{y}}}_i \in hE_{\delta,-} \Bigr) \, \prod_{i \in J} f(s_i) \, ds_i.
\end{align*}
Using first~\eqref{eq: densitytail} and then (O1), it follows that
\begin{align*}
&   P\Bigl[ \,\sum_{i \in J} S_i {\hat {\mathbf{y}}}_i \in hE_{\delta,-} \Bigr]  
  \sim \frac{C_\alpha^k}{2^k}  
  \int_{ s_1, \ldots, s_k \in \mathbb{R} \colon \atop |s_1|,\ldots, |s_k | > \varepsilon_0 h}
I \Bigl( \sum_{i \in J} s_i {\hat {\mathbf{y}}}_i \in hE_{\delta,-} \Bigr) \, \prod_{i \in J} \alpha s_i^{-(1+ \alpha)} \, ds_i
\\&\qquad = \frac{C_\alpha^k}{2^k} \int_{\mathbb{R}^k}
I \Bigl( \sum_{i \in J} s_i {\hat {\mathbf{y}}}_i \in hE_{\delta,-} \Bigr) \,  \prod_{i \in J} \alpha s_i^{-(1+ \alpha)} \, ds_i.
\end{align*}
Making the change of variables \(   (2 \Lambda(\mathbf{y}_i))^{1/\alpha}s_i /h \mapsto s_i\), we obtain
\begin{align*}
&\frac{C_\alpha^k}{2^k}   \cdot \biggl[ \, \prod_{i \in J} 2\Lambda(\mathbf{y}_i )h^{-\alpha} \biggr] \int_{\mathbb{R}^k}
I \Bigl( \sum_{i \in J} s_i { {\mathbf{y}}}_i \in E_{\delta,-} \Bigr) \,  \prod_{i \in J} \alpha s_i^{-(1+ \alpha)} \, ds_i  =
\\&\qquad 
 C_\alpha^k h^{-k \alpha}  
\biggl[\, \prod_{i \in J} \Lambda(\mathbf{y}_i ) \biggr] \int_{\mathbb{R}^k}
 I\Bigl( \sum_{i \in J} s_i { {\mathbf{y}}}_i \in E_{\delta,-} \Bigr) \, \prod_{i \in J} \alpha s_i^{-(1+ \alpha)} \, ds_i.
\end{align*}
Summing over all \( J \subseteq [m] \) with \( |J| = k \), we  get
\begin{align*}
& C_\alpha^k h^{-k \alpha}  
\sum_{J \subseteq [m] \colon |J|=k} 
\Biggl[
 \biggl[ \, \prod_{i \in J} \Lambda(\mathbf{y}_i ) \biggr] \int_{\mathbb{R}^k}
 I\Bigl( \sum_{i \in J} s_i { {\mathbf{y}}}_i \in E_{\delta,-} \Bigr) \, \prod_{i \in J} \alpha s_i^{-(1+ \alpha)} \, ds_i \Biggr]
\\&\qquad =
C_\alpha^k h^{-k \alpha}  
\int_{\mathbb{R}^k}
  \sum_{J \subseteq [m] \colon |J|=k} 
 \Biggl[ \biggl[ \,  \prod_{i \in J} \Lambda(\mathbf{y}_i ) \biggr] \, 
 I \Bigl( \sum_{i \in J} s_i { {\mathbf{y}}}_i \in E_{\delta,-} \Bigr) \,  \prod_{i \in JX} \alpha s_i^{-(1+ \alpha)} \, ds_i \Biggr].
\end{align*}
Now note that 
\begin{enumerate}[(i)]
	\item each pair of points, \( \pm \mathbf{y}_i \), \( i = 1,2,\ldots , m \), is counted only once in the 
last equation  and
	\item each set \( J \) of size \( k \) can be ordered exactly in \( k! \) ways.
\end{enumerate}
Using this, and symmetry, it follows that the previous equation is equal to
\begin{align*}
& 
\frac{ C_\alpha^k h^{-k \alpha}}{2^k k!}  
\int_{\mathbb{R}^k}
  \Lambda^k \Biggl( \biggl\{  \mathbf{x}_1, \ldots, \mathbf{x}_k \in \mathbb{S}^{n-1} \colon 
 \sum_{i =1}^k s_i { {\mathbf{x}}}_i \in E_{\delta,-} \biggl\} \Biggl)\,  \prod_{i =1}^k \alpha s_i^{-(1+ \alpha)} \, ds_i=
 \\&\qquad
\frac{ C_\alpha^k h^{-k \alpha}}{k!}  
\int_{\mathbb{R}_+^k}
  \Lambda^k \Biggl( \biggl\{ \mathbf{x}_1, \ldots, \mathbf{x}_k \in \mathbb{S}^{n-1} \colon 
 \sum_{i =1}^k s_i { {\mathbf{x}}}_i \in E_{\delta,-} \biggl\} \Biggl)\,  \prod_{i =1}^k \alpha s_i^{-(1+ \alpha)} \, ds_i
\end{align*}
and hence, by taking $h$ to infinity and then $\delta$ to zero,
\begin{align*}
&\liminf_{h \to \infty} h^{k \alpha} P(X \in hE)\nonumber
\\&\qquad \geq \lim_{\delta \to 0} 
\frac{ C_\alpha^k}{k!}  
\int_{\mathbb{R}^k_+}
  \Lambda^k \Biggl( \biggl\{ \mathbf{x}_1, \ldots, \mathbf{x}_k \in \mathbb{S}^{n-1} \colon 
 \sum_{i =1}^k s_i { {\mathbf{x}}}_i \in E_{\delta,-} \biggl\} \Biggl)\,  \prod_{i =1}^k \alpha s_i^{-(1+ \alpha)} \, ds_i.
 \end{align*}
Using the monotone convergence theorem, this implies in particular that
 \begin{align*}
&\liminf_{h \to \infty} h^{k \alpha} P(X \in hE)\nonumber
\geq L(E^o,k,\alpha)
 \end{align*}
and hence the lower bound in Theorem~\ref{theorem: later tails} holds. The proof of the upper bound is 
completely analogous and slightly easier, and is hence omitted here.

\paragraph{Step 2.}
It now remains only to show that the assumption on \( \bar E \) given in~\eqref{eq: the k point condition} can 
be removed. So we now assume that
\begin{equation*}
\exists ( s_1,\ldots, s_m) \in \mathbb{R}^m  \colon \sum_{i=1}^m s_i {\hat {\mathbf{y}}}_i \in  \bar E \text{ and }  | \{ i \in [m] \colon s_i \not = 0 \}| < k.
\end{equation*}
 Then it is easy to see that the integral in the definition of \( L(E_{\delta,+},k,\alpha)\) is infinite for every \( \delta > 0\) and hence the upper bound  holds without the assumption on \( \bar E \).

We now show that the lower bound holds also without the assumption on \( \bar E \). To this end,  assume first that there is \(  \mathbf{t} \coloneqq ( t_1,\ldots, t_m) \in \mathbb{R}^m  \) such that 
\begin{enumerate}[(i)]
\item \( \sum_{i=1}^m t_i {\hat {\mathbf{y}}}_i \in    E^o  \), and
\item \( | \{ i \in [m] \colon t_i \not = 0 \}| = \ell <k \).
\end{enumerate}
Assume further that \( \ell \) is the smallest integer for which such a point \( \mathbf{t} \) exists. Then, for all sufficiently small \( \delta > 0 \) we have that  \( \sum_{i=1}^m t_i {\hat {\mathbf{y}}}_i  \in E_{\delta, -} \), and
\begin{equation*}
\forall ( t_1,\ldots, t_m) \in \mathbb{R}^m  \colon \sum_{i=1}^m t_i {\hat {\mathbf{y}}}_i \in  \overline {E_{\delta,-}} \Rightarrow | \{ i \in [m] \colon t_i \not = 0 \}| \geq \ell.
\end{equation*}
Since by the first part of the proof, we have that  
\[
\liminf_{h \to \infty} h^{\ell \alpha } P(X \in hE) \geq \liminf_{h \to \infty} h^{\ell \alpha } P(X \in hE_{\delta,-}) = L(E_{\delta,-},\ell,\alpha) > 0
\]
it follows that
\[
\liminf_{h \to \infty} h^{k \alpha } P(X \in hE) = \infty  
\]
and hence the lower bound is still valid in this case. If no such point \( \mathbf{t} \) exists, then we have that
\[
\forall ( t_1,\ldots, t_m) \in \mathbb{R}^m  \colon \sum_{i=1}^m t_i {\hat {\mathbf{y}}}_i \in   E^o \Rightarrow  | \{ i \in [m] \colon t_i \not = 0 \}| \geq k.
\]
Using Step 1, this implies in particular that for all \( \delta >  0 \), we have that
\[  
\liminf_{h \to \infty} h^{k \alpha} P(X \in hE) \geq   \liminf_{h \to \infty}  h^{k \alpha} P(X \in hE_{\delta,-}) =   L(E_{\delta,-}, k, \alpha).
\]
Since \( L(E_{\delta,-}, k, \alpha) \) is monotone in \( \delta \), the desired conclusion follows by applying the monotone convergence theorem.
This concludes the proof.

\end{proof}

\begin{remark}
We observe that we have shown that if there is a matrix 
\( A = ( {\hat {\mathbf{y}}}_1 , {\hat {\mathbf{y}}}_2 , \ldots, {\hat {\mathbf{y}}}_m ) \) such 
that  \( X  \overset{\mathcal{D}}{=} A (S_1,\ldots, S_m) \), where \( S_1, S_2, \ldots , S_m \sim S_\alpha \)  
are i.i.d.\ \!\! (or equivalently that the spectral measure is finitely supported), 
then, for any set \( E \subseteq \mathbb{R}^n \), 
\begin{align*}
&\frac{1}{k!}\int_0^\infty
 \cdots
\int_{0}^\infty
\Lambda^k \Biggl( \biggl\{ \mathbf{x}_1, \ldots, \mathbf{x}_k \in \mathbb{S}^{n-1} \colon   
\sum_{i=1}^k s_i\mathbf{x}_i \in  E  \biggl\} \Biggl)\, 
    \prod_{i=1}^k \alpha  s_i^{-(1+\alpha)}\, ds_i
\\&\qquad =
2^{-k} \!\!\!\!\sum_{J \subseteq [m] \colon | J | = k}
\int_\mathbb{R}
 \cdots
\int_\mathbb{R}
I\biggl(  
\sum_{i \in J} s_i
{\hat {\mathbf{y}}}_i \in  E   \biggr)\, 
 \prod_{i=1}^k \alpha  s_i^{-(1+\alpha)}\, ds_i.
\end{align*}
\end{remark}

\begin{remark}
With only small adjustments of the proof above, the assumption that \( X \) is symmetric can be dropped. To do this, one replaces the matrix representation used above with the corresponding representation for when \( X \) is not symmetric (i.e.\ define \( A \) by \( A(\cdot, i) = (\Lambda(\mathbf{y}_i))^{1/\alpha} \mathbf{y}_i \) and \( S_i \) is a so-called totally skewed \( \alpha \)-stable random variable with scale one, and then adjust the proof accordingly. This is not as easy to do however when \( \Lambda \) is not finitely supported.
\end{remark}

\begin{remark}
By Theorem 1(ii) in~\cite{bnr1993}, any multivariate stable distribution \( X\sim S_\alpha(\Lambda) \)   can be approximated by a multivariate stable distribution \( X_\varepsilon \sim S_\alpha(\Lambda_\varepsilon) \) which is such that 
\begin{enumerate}[(i)]
\item \( \Lambda_\varepsilon \) is finitely supported, and
\item \mbox{} \vspace{-3ex}  \[   \sup_{E \colon E \subseteq \mathbb{R}^n,\atop E \text{ is a Borel set}}  \bigl| P(X \in E) - P(X_\varepsilon \in E) \bigr| < \varepsilon . \]
\end{enumerate}
Here \( \Lambda_\varepsilon \) is chosen by partitioning the unit sphere into a finite number of sets of small diameter, and then concentrating all the mass of \( \Lambda \) in each such  set at an arbitrarily chosen  point in the set.

This result, together with the proof for the finitely supported case, is however not sufficient to be able to make the same conclusion for any spectral measure.  To see this,  let \( E \) and \( \Lambda \) be as in Example~\ref{example: from paper}, and let \( \alpha \in (0,1) \) so that Example~\ref{example: from paper} gives that \( \lim_{h \to \infty} h^{2 \alpha}P(X \in h E ) \in (0,\infty)\).  Then there are \( \Lambda_\varepsilon \) as above which are arbitrarily close to \( \Lambda \) but for which the corresponding limit is infinite by Theorem~\ref{theorem: later tails}. 
\end{remark}

To be able to give the proof of Theorem~\ref{theorem: later tails} in 
the general setting, we will first need the following lemma.
The special case \( k = 2 \) was stated in~\cite{st1994} (see Equation 1.4.8 on p.\ 27), 
but no proof is given there. A sketch of the proof of this particular case was provided in private 
correspondence with one of the authors.

\begin{lemma}\label{lemma: moments of tails}
Let \( (W_i )_{i \geq 1} \) be a sequence of i.i.d.\ random variables with  \( 0\le W_i \leq 1\), 
\( (\varepsilon_i )_{i \geq 1} \) be a sequence of i.i.d.\ random variables with 
\( \varepsilon_i\sim \unif ( \{ -1,1 \}) \) and  \( (\Gamma_i)_{i \geq 1} \) 
be the arrival times of a Poisson process with rate one where we
assume that these three sequences are independent of each other.
Next let \( \alpha \in (0,2) \), \( k \geq 2 \) be an integer
and \( \varepsilon \in (0, \min(\{ \alpha, (k-1)(2-\alpha) \})) \). Then
\[
\E \Biggl[ \biggl| \sum_{i = k}^\infty \varepsilon_i \Gamma_i^{-1/\alpha} W_i \biggr|^{(k-1)\alpha + \varepsilon} \Biggr] < \infty.
\] 
\end{lemma}

\begin{proof}[Proof of Lemma~\ref{lemma: moments of tails}]
To simplify notation, write    \( \beta \coloneqq (k-1)\alpha + \varepsilon \). We then need to show that 
\[
\E \Biggl[ \biggl| \sum_{i = k}^\infty \varepsilon_i \Gamma_i^{-1/\alpha} W_i \biggr|^{\beta} \Biggr] < \infty.
\] 
To this end, note first that for any fixed \( m\geq k \)  we have that
\begin{equation*}\label{eq: first moment bound on finite stable sum}
\begin{split}
&\E \Biggl[ \biggl| \sum_{i=k}^m \varepsilon_i \Gamma_i^{-1/\alpha} W_i \biggr|^\beta \Biggr] \leq \E \Biggl[ \biggl[ \, \sum_{i=k}^m \Gamma_i^{-1/\alpha} W_i \biggr]^\beta \Biggr]
\leq \E \Biggl[ \biggl[ \,   \sum_{i=k}^m   \Gamma_i^{-1/\alpha}   \biggr]^\beta \Biggr]
\\&\qquad \leq \E \biggl[ \Bigl(   (m-k+1) \Gamma_k^{-1/\alpha}   \Bigr)^\beta \biggr]
=   (m-k+1)^\beta \,  \E \Bigl[  \Gamma_k^{-\beta/\alpha}   \Bigr].
\end{split}
\end{equation*}
Since \( k > \beta/ \alpha \), we have that \( \E \Bigl[ \Gamma_k^{-\beta/\alpha}\Bigr] < \infty \),
and hence 
\begin{equation}\label{eq: sum of first terms is finite}
\E \Biggl[ \biggl| \sum_{i=k}^m \varepsilon_i \Gamma_i^{-1/\alpha} W_i \biggr|^\beta \Biggr] < \infty
\end{equation}
for any fixed \( m \geq k \).

Now recall that  for any real-valued random variables \( X \) and \( Y \) with \( \E \left[ |X|^\beta \right]<\infty \) and \( \E \left[ |Y|^\beta \right]<\infty \) we have that \( \E [ |X+Y|^\beta ] < \infty.\) Using~\eqref{eq: sum of first terms is finite}, the conclusion of the lemma will thus follow if we can prove that
\begin{align*}
\E \Biggl[ \biggl| \sum_{i=m}^\infty \varepsilon_i \Gamma_i^{-1/\alpha} W_i \biggr|^{\beta} \Biggr] < \infty
\end{align*}
for some \( m \geq k\).
To this end, fix  \( m >\beta/\alpha \cdot k /(k-1) \).  Then
\begin{align*}
&\E \Biggl[ \biggl| \sum_{i = m}^\infty \varepsilon_i \Gamma_i^{-1/\alpha} W_i \biggr|^\beta \Biggr]  =
\E \Biggl[   \biggl[ \biggl( \,  \sum_{i = m}^\infty \varepsilon_i \Gamma_i^{-1/\alpha} W_i \biggr)^{2(k-1)} \biggr]^{\beta/(2(k-1))} \Biggr]  
\\&\qquad =
\E \Biggl[ \E_{(\varepsilon_i)}\Biggl[  \biggl( \Bigl( \sum_{i = m}^\infty \varepsilon_i \Gamma_i^{-1/\alpha} W_i \Bigr)^{2(k-1)} \biggr)^{\beta/(2(k-1))} \mid \sigma\bigl( (W_i),\, (\Gamma_i)\bigr)  \Biggr]  \Biggr].
\end{align*}
Since \( \beta/(2(k-1)) = ((k-1)\alpha + \varepsilon)/(2(k-1))< 1 \) by the assumption on \( \varepsilon \), we can apply Jensen's inequality to bound this expression from above by
\begin{align*}
&\E \Biggl[ \E_{(\varepsilon_i)}\biggl[ \Bigl( \sum_{i = m}^\infty \varepsilon_i \Gamma_i^{-1/\alpha} W_i \Bigr)^{2(k-1)}  \mid \sigma\bigl( (W_i),\, (\Gamma_i) \bigr) \biggr]^{\beta/(2(k-1))}  \Biggr]
\\&\qquad \leq
\E \Biggl[  \biggl[ (2(k-1))! \!\!\!\sum_{i_1, \ldots, , i_{k-1} \colon \atop m \leq i_1 \leq  \ldots \leq  i_{k-1}} \!\prod_{j=1}^{k-1}   \Gamma_{i_j}^{-2/\alpha}  W_{i_j}^2   \biggr]^{\beta/(2(k-1))}  \Biggr].
\end{align*}
Now we can again use the fact that \( \beta/(2(k-1))<1 \) and the so-called 
\( c_r \)-inequality (see e.g.\ Theorem 2.2 in \cite{gut}) to move this exponent into the summands to bound the previous 
expression  from above by
\begin{align*}
&((2(k-1))!)^{\beta/(2(k-1))}\; \E \Biggl[ \,  \sum_{i_1, \ldots, , i_{k-1} \colon \atop m \leq i_1 \leq  \ldots \leq  i_{k-1}} \!\!\!\!\! \Gamma_{i_j}^{-\beta/(\alpha(k-1))}  W_{i_j}^{\beta/(k-1)}  \Biggr]
\\&\qquad =
 ((2(k-1))!)^{\beta/(2(k-1))} \!\!\!\!\!\!\!\! \sum_{i_1, \ldots, , i_{k-1} \colon \atop m \leq i_1 \leq  \ldots \leq  i_{k-1}} \!\!\!\!\!  \E \Biggl[ \, \prod_{j=1}^{k-1}   \Gamma_{i_j}^{-\beta/(\alpha(k-1))}  W_{i_j}^{\beta/(k-1)}  \Biggr]
 \\&\qquad \leq
 ((2(k-1))!)^{\beta/(2(k-1))} \!\!\!\!\!\!\!\!  \sum_{i_1, \ldots, , i_{k-1} \colon \atop m \leq i_1 \leq  \ldots \leq  i_{k-1}} \!\!\!\!\! \E \Biggl[ \, \prod_{j=1}^{k-1}   \Gamma_{i_j}^{-\beta/(\alpha(k-1))}    \Biggr].
\end{align*}
In particular, this implies that it now only remains to show that
\begin{equation}\label{eq: the rest of the sum}
\sum_{i_1, \ldots, , i_{k-1} \colon \atop m \leq i_1 \leq  \ldots \leq  i_{k-1}} \!\!\!\!\! \E \Biggl[\, \prod_{j=1}^{k-1}   \Gamma_{i_j}^{-\beta/(\alpha(k-1))}    \Biggr] < \infty.
\end{equation}
 
To do this, first fix  \( \gamma \in \mathbb{R}_+ \). If \(  i \in \mathbb{Z}_+ \), then 
\(\E [\Gamma_i^{-\gamma}]<\infty \) if and only if \( i > \gamma \). 
Moreover, for such \( i \) and \( \gamma \) we easily have that
\(\E \left[ \Gamma_i^{-\gamma}\right]=\frac{\Gamma(i-\gamma)}{\Gamma(i)}\) .
By Stirling's formula, it follows that for a fixed \( \gamma \) we have that \( \E \left[ \Gamma_i^{-\gamma}\right]  \sim i^{-\gamma} \) and hence \( \E \left[ \Gamma_i^{-\gamma}\right] < C_\gamma i^{-\gamma} \) for some constant \( C_\gamma \ge 1 \) and all \( i > \gamma \).

Now assume that \( 1 \leq i_1 \leq \ldots \leq i_{k-1} \) is a sequence of integers.  
Then for \( j=2,3, \ldots, k-1 \) we have that \( \Gamma_{i_j} \geq \Gamma_{i_1} \), and 
\( \Gamma_{i_j} \geq \Gamma_{i_j} - \Gamma_{i_{j-1}} \). The random variables 
\( \Gamma_{i_j} - \Gamma_{i_{j-1}}  \)  are independent and equal in distribution to 
\( \Gamma_{i_j-i_{j-1}} \) if \( i_j \not = i_{j-1} \) noting that \(\Gamma_0=0 \).
Using this, it follows that for any fixed integer \( M > 0 \) we have that
\begin{align*}
\E \biggl[ \, \prod_{j=1}^{k-1} \Gamma_{i_j}^{-\gamma}\biggr] 
&\leq  \E \biggl[ \Gamma_{i_1}^{-\gamma - \gamma \sum_{j \geq 2}^{k-1} \mathbf{1}(i_j - i_{j-1} < M)} \biggr] \; \cdot \;  \prod^{k-1}_{j \geq 2 \colon \atop \mathclap{ i_j-i_{j-1} \geq M}}  \E \Bigl[ \Gamma_{i_j-i_{j-1}}^{ -\gamma}  \Bigr] 
\\&\leq  \E \Bigl[ \Gamma_{i_1}^{-\gamma  } \cdot I(\Gamma_{i_1} \geq 1) + \Gamma_{i_1}^{-(k-1)\gamma}\cdot I(\Gamma_{i_1} < 1) \Bigr] \; \cdot \; \prod^{k-1}_{j \geq 2 \colon \atop \mathclap{ i_j-i_{j-1} \geq M}}  \E \Bigl[ \Gamma_{i_j-i_{j-1}}^{ -\gamma} \Bigr]
\\&\leq  \E \Bigl[ \Gamma_{i_1}^{-\gamma  }  + \Gamma_{i_1}^{-(k-1)\gamma}  \Bigr] \; \cdot \; \prod^{k-1}_{j \geq 2 \colon \atop \mathclap{ i_j-i_{j-1} \geq M}}  \E \Bigl[ \Gamma_{i_j-i_{j-1}}^{ -\gamma} \Bigr].
\end{align*}
If \( i_1 > (k-1)\gamma \) and \( M > \gamma \), then, using the above, 
this is bounded from above by
\begin{align*}
(C_\gamma+C_{(k-1)\gamma})\cdot   {i_1}^{-\gamma}  \cdot 
C_\gamma^{k-2}  \prod^{k-1}_{j \geq 2 \colon \atop \mathclap{i_j-i_{j-1} \geq M}}   (i_j-i_{j-1})^{ -\gamma} .
\end{align*}

In particular, if we let  \( \gamma = \beta/(\alpha(k-1)) = (\alpha(k-1) + \varepsilon)/(\alpha(k-1))> 1 \) and \( M = m \), then for \( i_1 \geq m \) we have that
\[
i_1 \geq m > \beta/\alpha \cdot k/(k-1)  = k \gamma > (k-1)\gamma
\]
and therefore, since $k\ge 2$, 
\(
M = m > \gamma.
\)
Hence it follows that~\eqref{eq: the rest of the sum} is bounded from above by 
\begin{align*}
  &\bigl(C_{\beta/(\alpha(k-1))} +C_{\beta/\alpha} \bigr) \, C^{k-1}_{\beta/(\alpha(k-1))} 
   \!\!\!\!\!\!
  \sum_{i_1, \ldots, , i_{k-1} \colon \atop m \leq i_1 \leq  \ldots \leq  i_{k-1}} \!\!\!\!\!\!\!
  i_1^{- \beta/(\alpha(k-1))} 
  \!\!\!\!\prod_{\substack{j \in \{ 2,3, \ldots, k-1 \} \colon \\i_{j}-i_{j-1} \geq m}} \!\!\!\!\!\! (i_j - i_{j-1})\mathrlap{{}^{- \beta/(\alpha(k-1))}}.
\end{align*}
This implies in particular that it only remains to show that
\[
\sum_{i_1, \ldots, , i_{k-1} \colon \atop m \leq i_1 \leq  \ldots \leq  i_{k-1}} \!\!\!\!\! i_1^{-\beta/(\alpha(k-1))} \prod_{\substack{j \in \{ 2,3, \ldots, k-1 \} \colon\\  i_{j}-i_{j-1} \geq m}}\!\!\!\!\!\!  (i_j - i_{j-1})^{-\beta/(\alpha(k-1))}< \infty.
\]
To see this, we first change the order of summation as follows. First, we will sum over all possible 
choices of \( i_1 \). Then we sum over the number \( G \) of terms in the product, which will 
range between \( 0 \) and \( k-2 \). Finally, we sum also over the possible choices of 
\( \ell_j \coloneqq i_j - i_{j-1} \) in the product, which will range from \( m \) to infinity. 
To sum over all possible sequences \( m \leq i_1 \leq \ldots \leq i_{k-1} \), we find an upper bound 
on the number of ways to choose the differences \( i_j - i_{j-1}  \) which are smaller than \( m \) 
and also, on the number of ways to choose which of the differences are larger than or equal to 
\( m \). The former 
of these quantities is clearly bounded from above by \( m^{k-2} \), and the latter is equal 
to \( \binom{k-2}{G} < 2^{k-2}\). Putting these observations together, we get
\begin{align*}
&\sum_{i_1, \ldots, , i_{k-1} \colon \atop m \leq i_1 \leq  \ldots \leq  i_{k-1}} \!\! \prod_{j \in \{ 2,3, \ldots, k-1 \} \colon \atop \mathclap{i_{j}-i_{j-1} \geq m}} \!\!\!\!\!\!\!\!\!\! (i_j - i_{j-1})^{-\beta/(\alpha(k-1))}
\\&\qquad  \leq 
(2m)^{k-2}
\sum_{i_1 = m}^\infty  i_1^{-\beta/(\alpha(k-1))} 
\sum_{G = 0}^{k-2} 
\; \sum_{\ell_1,\ldots, \ell_G =m}^\infty 
\; \prod_{j=1}^G \ell_j^{-\beta/(\alpha(k-1))} 
\\&\qquad  = 
(2m)^{k-2}\sum_{i_1 = m}^\infty  i_1^{-\beta/(\alpha(k-1))} \sum_{G = 0}^{k-2}   \Biggl( \, \sum_{\ell =m}^\infty  \ell^{-\beta/(\alpha(k-1))} \Biggr)^G
\\&\qquad  = 
(2m)^{k-2}  \sum_{G = 1}^{k-1}   \Biggl( \,  \sum_{\ell =m}^\infty  \ell^{-\beta/(\alpha(k-1))} \Biggr)^G\!\!\!\!.
\end{align*}
Since \( \beta/(\alpha(k-1)) = (\alpha(k-1) + \varepsilon)/(\alpha(k-1))>1 \), the desired conclusion now follows.

\end{proof}

We now state the following lemma which will be used in the proof of Theorem~\ref{theorem: later tails}. 
For a proof of this lemma we refer the reader to~\cite{st1994}.

\begin{lemma}[Theorem 3.10.1 in~\cite{st1994}]\label{lemma: multivariate converging sum}
Let \( \Lambda \) be a symmetric spectral measure on \( \mathbb{S}^{n-1} \). 
Furthermore, let \( C_\alpha \) be defined by \(   P(Y \geq h) \sim C_\alpha h^{-\alpha}/2 \) for 
\( Y \sim S_\alpha \), let \( (\Gamma_i)_{i\geq 1} \) be the arrival times of a rate one 
Poisson process and let \( (\mathbf{W}_i)_{i\geq 1 } \) be i.i.d., each with distribution 
\( \bar \Lambda \coloneqq \Lambda/\Lambda(\mathbb{S}^{n-1}) \) (the normalized spectral measure), 
independent of the Poisson process.  Then
\[
C_\alpha^{1/\alpha} \Lambda(\mathbb{S}^{n-1})^{1/\alpha} \sum_{i=1}^\infty \Gamma_i^{-1/\alpha} \mathbf{W}_i
\] 
converges almost surely to a random vector with distribution \( S_\alpha(\Lambda) \).
\end{lemma}

We now give a proof of Theorem~\ref{theorem: later tails} using 
Lemmas~\ref{lemma: moments of tails}~and~\ref{lemma: multivariate converging sum}.

\begin{proof}[Proof of Theorem~\ref{theorem: later tails}]
Let  \( C_\alpha\),  \( (\Gamma_i) \) and \( ( \mathbf{W}_i) \) be as in 
Lemma~\ref{lemma: multivariate converging sum}.
Define 
\[
X = \big(X_1,X_2, \ldots, X_n \bigr) \coloneqq C_\alpha^{1/\alpha} \Lambda \bigl(\mathbb{S}^{n-1}\bigr)^{1/\alpha} \sum_{i=1}^\infty \Gamma_i^{-1/\alpha} \mathbf{W}_i .
\]
Then Lemma~\ref{lemma: multivariate converging sum} implies that \( X \) has distribution 
\(S_\alpha(\Lambda) \). For \( j \in [n] \) and \( i \geq 1 \), let \( \mathbf{W}_i(j) \) denote the \( j \)th component of \( \mathbb{W}_i \).  By Markov's inequality, for any \( j = 1,2, \ldots, n \)  and all 
\( h > 0 \) and \( \varepsilon > 0 \) we have that
\begin{align*}
&P \biggl[ \Bigl( C_\alpha \Lambda\bigl( \mathbb{S}^{n-1}\bigr)\Bigr)^{1/\alpha} \, \Bigl| \sum_{i = {k+1}}^\infty  \Gamma_i^{-1/\alpha} \mathbf{W}_i(j) \Bigr| > h \biggr]
\\&\qquad \leq \frac{\E \left[\Bigl( C_\alpha \Lambda\bigl( \mathbb{S}^{n-1}\bigr)\Bigr)^{(k\alpha+\varepsilon)/\alpha} \, \Bigl| \sum_{i = k+1}^\infty  \Gamma_i^{-1/\alpha} \mathbf{W}_i(j) \; \Bigr|^{k \alpha + \varepsilon} \right]}{h^{k\alpha + \varepsilon}} .
\end{align*}
By picking \( \varepsilon\) sufficiently small and applying Lemma~\ref{lemma: moments of tails} using $k+1$
(noting that by symmetry, $W_i(j)$ has the same distribution as $\epsilon_i |W_i(j)|$ with the two
factors independent), it follows that
\[
P \biggl[\Bigl( C_\alpha \Lambda\bigl( \mathbb{S}^{n-1}\bigr)\Bigr)^{1/\alpha}  \Bigl| \sum_{i = k+1}^\infty  \Gamma_i^{-1/\alpha} \mathbf{W}_i(j)\Bigr| > h \biggr] \leq o\bigl(h^{-k\alpha}\bigr)
\]
and hence
\[
P \biggl[\Bigl( C_\alpha \Lambda\bigl( \mathbb{S}^{n-1}\bigr)\Bigr)^{1/\alpha}  \, \Bigl\| \sum_{i = k+1}^\infty   \Gamma_i^{-1/\alpha} \mathbf{W}_i\Bigr\|_\infty > h \biggr] \leq o\bigl(h^{-k \alpha}\bigr).
\]
This implies in particular that for any \( \varepsilon' > 0 \)
\[
P \biggl[ \Bigl( C_\alpha \Lambda\bigl( \mathbb{S}^{n-1}\bigr)\Bigr)^{1/\alpha}  \, \Bigl\| \sum_{i = k+1}^\infty   \Gamma_i^{-1/\alpha} \mathbf{W}_i\Bigr\|_\infty > \varepsilon' h \biggr] \leq o\bigl(h^{-k \alpha}\bigr).
\]

Now for any $\delta >0$, let \( E_{\delta,-} \coloneqq 
\bigl\{ \mathbf{x} \in E \colon d(x,\partial E)> \delta \bigr\} \). Setting
$\delta\coloneqq \sqrt{n} \varepsilon'$, we then have
\begin{equation}\label{eq: h prime inequality I}
\begin{split}
 & P\bigl( (X_1,X_2,\ldots, X_n) \in hE  \bigr) 
\\&\qquad=
P\biggl[ (X_1,X_2,\ldots, X_n) \in hE, \,
\Bigl( C_\alpha \Lambda\bigl( \mathbb{S}^{n-1}\bigr)\Bigr)^{1/\alpha}  \, \Bigl\| \sum_{i = k+1}^\infty   \Gamma_i^{-1/\alpha} \mathbf{W}_i\Bigr\|_\infty < \varepsilon' h\biggr] + o(h^{-k\alpha})  
\\& \qquad \geq 
P\Biggl[ \Bigl( C_\alpha \Lambda\bigl( \mathbb{S}^{n-1}\bigr)\Bigr)^{1/\alpha} \Bigl(   \Gamma_1^{-1/\alpha} \mathbf{W}_1+ \ldots +    \Gamma_k^{-1/\alpha} \mathbf{W}_k   \Bigr) \in   hE_{\delta,-}, \,
\\&\qquad\qquad \qquad
\Bigl( C_\alpha \Lambda\bigl( \mathbb{S}^{n-1}\bigr)\Bigr)^{1/\alpha}  \,  \Bigl\| \sum_{i = k+1}^\infty   \Gamma_i^{-1/\alpha} \mathbf{W}_i\Bigr\|_\infty < \varepsilon' h
\Biggr] + o(h^{-k\alpha})
\\&\qquad  =P\Biggl[ \Bigl( C_\alpha \Lambda\bigl( \mathbb{S}^{n-1}\bigr)\Bigr)^{1/\alpha} \left(   \Gamma_1^{-1/\alpha} \mathbf{W}_1+ \ldots +    \Gamma_k^{-1/\alpha} \mathbf{W}_k   \right) \in hE_{\delta,-}\Biggr]  + o(h^{-k\alpha}).
\end{split}
\end{equation}
Similarly, we have that
\begin{equation}\label{eq: h prime inequality II}
\begin{split} 
&
P\bigl( (X_1,X_2,\ldots, X_n) \in hE\bigr)  
\\& \qquad=
P\biggl[ (X_1,X_2,\ldots, X_n) \in hE,\,
\Bigl( C_\alpha \Lambda\bigl( \mathbb{S}^{n-1}\bigr)\Bigr)^{1/\alpha} \,  \Bigl\| \sum_{i = k+1}^\infty   \Gamma_i^{-1/\alpha} \mathbf{W}_i\Bigr\|_\infty < \varepsilon' h\biggr] + o(h^{-k\alpha}) 
\\&\qquad \leq
P\Biggl[
 \Bigl( C_\alpha \Lambda\bigl( \mathbb{S}^{n-1}\bigr)\Bigr)^{1/\alpha}  \left(   \Gamma_1^{-1/\alpha} \mathbf{W}_1+ \ldots +    \Gamma_k^{-1/\alpha} \mathbf{W}_k   \right) \in hE_{\delta,+},
 \\& \qquad\qquad\qquad 
\Bigl( C_\alpha \Lambda\bigl( \mathbb{S}^{n-1}\bigr)\Bigr)^{1/\alpha}  \, \Bigl\| \sum_{i = k+1}^\infty   \Gamma_i^{-1/\alpha} \mathbf{W}_i\Bigr\|_\infty < \varepsilon' h
\Biggr] + o(h^{-k\alpha})  
\\&\qquad\qquad \le P\biggl[ \Bigl( C_\alpha \Lambda\bigl( \mathbb{S}^{n-1}\bigr)\Bigr)^{1/\alpha} \left(   \Gamma_1^{-1/\alpha} \mathbf{W}_1+ \ldots +    \Gamma_k^{-1/\alpha} \mathbf{W}_k   \right) \in   h E_{\delta,+}\biggr] + o(h^{-k\alpha}).
\end{split}
\end{equation}
To be able to simplify these expressions, first recall that if 
\( (\Gamma_1, \Gamma_2, \ldots, \Gamma_{k+1} ) \) 
are the first \( k +1 \) arrivals of a mean one Poisson process and 
\( U_1,U_2, \ldots, U_k \sim \unif(0,1) \) are independent, 
\[
\bigl\{ \Gamma_1/\Gamma_{k+1} , \ldots, \Gamma_{k}/\Gamma_{k+1} \mid \Gamma_{k+1} \bigr\} 
\overset{d}{=} \bigl\{ U_1, \ldots, U_{k} \bigr\}.
\]
Using this and now letting \( U_1,U_2, \ldots, U_k \) be i.i.d.\ uniforms defined on the same 
probability space as everything else but independent of them,
we see that for \( E_{\delta,\cdot} = E_{\delta,+} \) or \( E_{\delta,\cdot} = E_{\delta,-} \), we have that
\begin{align*}
& P\biggl[ \Bigl( C_\alpha \Lambda\bigl( \mathbb{S}^{n-1}\bigr)\Bigr)^{1/\alpha}  \sum_{i=1}^k    \Gamma_i^{-1/\alpha} \mathbf{W}_i  \in  h E_{\delta,\cdot}\biggr] 
\\& \qquad =  P\biggl[ \Bigl( C_\alpha \Lambda\bigl( \mathbb{S}^{n-1}\bigr)\Bigr)^{1/\alpha} \Gamma_{k+1}^{-1/\alpha} \sum_{i=1}^k    U_i^{-1/\alpha} \mathbf{W}_i  \in h E_{\delta,\cdot} \biggr] 
\\& \qquad =  \int_0^\infty \frac{x^k e^{-x}}{k!} \int_0^1 \cdots \int_0^1 \bar \Lambda^k \Biggl( \biggl\{ (w_i)_{i=1}^k \colon    x^{-1/\alpha} \sum_{i=1}^k \left( \frac{C_\alpha \Lambda(\mathbb{S}^{n-1})}{u_i} \right)^{1/\alpha}   w_i \in   h E_{\delta,\cdot}  \biggr\} \Biggr)  \,  \prod_{i=1}^k du_i  \, dx.
\end{align*} 
 If, for each fixed $x$, we make the change of variables 
 \[ s_i = x^{-1/\alpha} \left( \frac{C_\alpha \Lambda(\mathbb{S}^{n-1})}{h^\alpha u_i} \right)^{1/\alpha},
 \]
then this simplifies to
\begin{align*}
& \int_0^\infty \frac{ e^{-x}}{k!} \Bigl( C_\alpha \Lambda \bigl(\mathbb{S}^{n-1} \bigr) h^{-\alpha} \Bigr)^k 
\int_0^\infty \cdots \int_0^\infty
  \bar \Lambda^k \Biggl( \biggl\{ (w_i)_{i=1}^k  \colon      \sum_{i=1}^k s_i  w_i \in  E_{\delta,\cdot} \biggr\} \Biggr)  
  \\&\qquad\qquad \cdot  I\biggl[ \min_i s_i > \Bigl( \frac{C_\alpha \Lambda(\mathbb{S}^{n-1})}{h^\alpha x}\Bigr)^{1/\alpha} \biggr] \, \prod_{i=1}^k\alpha s_i^{-(1+\alpha)}\, ds_i\, dx
\\& \qquad = \frac{C_\alpha^k h^{-\alpha k}}{k!} 
\int_0^\infty \cdots \int_0^\infty
  \Lambda^k \Biggl( \biggl\{ (w_i)_{i=1}^k  \colon      \sum_{i=1}^k s_i  w_i \in  E_{\delta,\cdot}\biggr\} \Biggr)  \, 
  \Biggl[ \int_{\frac{C_\alpha \Lambda (\mathbb{S}^{n-1})}{h^\alpha \min_i s_i^\alpha}}^\infty   e^{-x}
  \, dx \Biggr] \,
\prod_{i=1}^k\alpha s_i^{-(1+\alpha)}\, ds_i .
\end{align*}
Note that the integral above is increasing in \( h \).
Combining the previous equation with~\eqref{eq: h prime inequality I} and~\eqref{eq: h prime inequality II} and applying the monotone convergence theorem, it follows that for any \( \delta > 0 \),
\begin{align}
&\label{eq: I1}
  \frac{ C_\alpha^k }{k!} 
\int_0^\infty \cdots \int_0^\infty
  \Lambda^k \Biggl( \biggl\{ (w_i)_{i=1}^k  \colon      \sum_{i=1}^k s_i  w_i \in E_{\delta,-}\biggr\} \Biggr)  \, 
\prod_{i=1}^k\alpha s_i^{-(1+\alpha)}\, ds_i
\\&\nonumber\qquad \leq
\liminf_{h \to \infty}  h^{k \alpha } P\bigl( (X_1,X_2,\ldots, X_n) \in h  E \bigr) 
\\&\qquad  \nonumber\leq
\limsup_{h \to \infty}  h^{k \alpha } P\bigl( (X_1,X_2,\ldots, X_n) \in h  E \bigr) 
\\&\nonumber \qquad \leq
  \frac{ C_\alpha^k }{k!} 
\int_0^\infty \cdots \int_0^\infty
  \Lambda^k \Biggl( \biggl\{ (w_i)_{i=1}^k  \colon      \sum_{i=1}^k s_i  w_i \in E_{\delta,+}\biggr\} \Biggr)  \, 
\prod_{i=1}^k\alpha s_i^{-(1+\alpha)}\, ds_i.
\end{align}
Noting that the integrand in~\eqref{eq: I1} is monotone in \( \delta \) and converges pointwise to the integrand in  \(L(E^o,k,\alpha)\), the desired conclusion follows  by letting \( \delta \to 0 \) and  applying the monotone convergence theorem. 
\end{proof}

 \section{Examples}\label{section: examples}

We will now apply Theorem~\ref{theorem: later tails} to a few examples.

\begin{example}\label{example: independent marginals}
Let \( \alpha \in (0,2) \) and let \( X_1 \) and \( X_2 \) be i.i.d.\ with 
\( X_1 \sim S_\alpha  \) and let \( X=(X_1,X_2) \). The corresponding spectral
measure \( \Lambda \) has four point masses each of weight \( 2^{-1} \) at 
\( (1,0) \), \( (0,1) \), \( (-1,0) \) and \( (0,-1) \).  With this example, we 
will consider three different sets $E$ which will be our three different cases.

\paragraph{Case (i)} Let \( E = \bigl\{ \mathbf{x} \in \mathbf{R}^2 \colon {x}_1,\, {x}_2 > 1 \bigr\} \). Then it is easy to see that $L(E^o,2,\alpha)=\lim_{\delta \to 0} L(E_{\delta,+},2,\alpha)$ and furthermore this common value is 
\begin{equation*}
C_\alpha^2 / 2 \cdot 2 \int_0^\infty \int_0^\infty (2^{-1} )^2 \cdot I \bigl(s_1,s_2 > 1\bigr) 
 \cdot \alpha s_1^{-(1 + \alpha)} \alpha s_2^{-(1+\alpha)} \, ds_1 \, ds_2
= C_\alpha^2 \cdot 2^{-2}.
\end{equation*}
Applying Theorem~\ref{theorem: later tails} with \( k = 2 \), we obtain
\begin{equation*}
\lim_{h \to \infty} h^{2\alpha} P \bigl(X_1,X_2 > h \bigr) = C_\alpha^2 \cdot 2^{-2}
\end{equation*}
which is of course consistent with what independence yields.

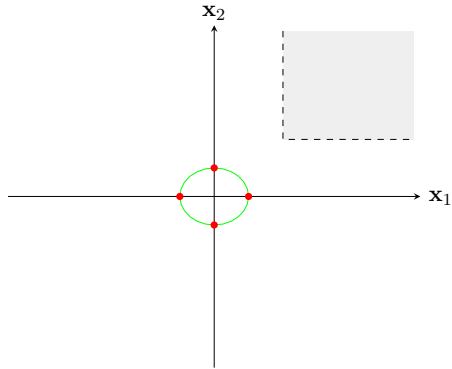
\begin{figure}[ht]

\centering
\begin{tikzpicture}[scale=0.8]
\begin{axis}[axis x line=middle,axis y line=middle, xmin=-3,xmax=3, ymin=-3,ymax=3, xlabel=$\mathbf{x}_1$, ylabel=$\mathbf{x}_2$, x label style={at={(axis description cs:1.05, 0.54)},anchor=north}, y label style={at={(axis description cs:0.5, 1.08)},anchor=north}, xtick={0}, ytick={0}]
\fill[fill = gray!30, opacity=.4] (axis cs: 1,1) rectangle (axis cs: 2.9,2.9);
\draw[dashed] (axis cs: 1,2.9) -- (axis cs: 1,1) -- (axis cs: 2.9,1);

\draw[green,domain=0:360] plot (axis cs: {0.5*cos(\x)}, {0.5*sin(\x)});

\fill[red] (axis cs: 0.5,0) circle (0.6mm); 
\fill[red] (axis cs: -0.5,0) circle (0.6mm); 
\fill[red] (axis cs: 0,0.5) circle (0.6mm); 
\fill[red] (axis cs: 0,-0.5) circle (0.6mm); 
\end{axis}
\end{tikzpicture} 

\caption{The figure above shows the set \( 2E \) (gray area) considered in Case (i) of
Example~\ref{example: independent marginals}  together with the four points (in red) 
at which \( \Lambda \) is supported.} 
\end{figure}

\paragraph{Case (ii)}
Let    \( A \subseteq \mathbb{S}^{1} \cap (\varepsilon,\infty)^2 \) 
for some \(\varepsilon>0 \) and define
\[
 C_A \coloneqq \bigl\{ \mathbf{x} \in \mathbb{R}^2 \colon \| \mathbf{x} \|_2 > 1 \text{ and } 
\mathbf{x}/ \|\mathbf{x}\|_2 \in A \bigr\}
\]
be the cone above $A$. Then we have the following.
\begin{proposition} \label{prop: case2example}
Let \( X \), \( A \) and \( C_A \) be as above, and assume that in addition to the above, the boundary of $A$ has zero (one-dimensional) measure. Then
\[
\lim_{h \to \infty} h^{2 \alpha} P \bigl(X \in hC_A\bigr)=
\frac{ C_\alpha^2  }{8}
\int_A \alpha (\cos \theta \sin \theta)^{-(1 + \alpha)} \, d\theta
\]
\end{proposition} 

\begin{proof}
We begin with the following computation which is valid for any set $A$
contained in  \(  \mathbb{S}^{1} \cap (\varepsilon,\infty)^2 \).
\begin{align}
& 
\frac{C_\alpha^2  }{2}
\int_0^\infty
\int_{0}^\infty
\Lambda^2 \bigl( \mathbf{x}_1,  \mathbf{x}_2 \in \mathbb{S}^{1} \colon     s_1\mathbf{x}_1  + s_2\mathbf{x}_2 \in C_A  \bigr) 
\cdot    \prod_{i=1}^2 \alpha  s_i^{-(1+\alpha)}\, ds_i \nonumber
\\&\qquad =
C_\alpha^2   
\int_0^\infty
\int_{0}^\infty
2^{-2} \cdot I\bigl(    ( s_1,s_2) \in C_A  \bigr) 
\cdot    \prod_{i=1}^2 \alpha  s_i^{-(1+\alpha)}\, ds_i\nonumber
\\&\qquad =
\frac{C_\alpha^2  }{4}
\int_{(s_1,s_2) \in C_A}    \alpha^2   s_1^{-(1+\alpha)}    s_2^{-(1+\alpha)}\, ds_1 \, ds_2 \label{eq: hidden variation}
\\&\qquad =
\frac{C_\alpha^2  }{4}
\int_A \int_1^\infty   \alpha^2   (r \cos \theta)^{-(1+\alpha)}    (r \sin \theta)^{-(1+\alpha)}\, r\,  dr \, d\theta\nonumber
\\&\qquad =
\frac{ C_\alpha^2  }{8}
\int_A \alpha (\cos \theta \sin \theta)^{-(1 + \alpha)} \, d\theta .\nonumber
\end{align}

For any set $U$, letting $U^o$ be the interior of $U$, one easily checks that
$$
(C_A)^o=C_{A^o} \mbox{ and }
C_{\bar A}\subset {\bar C_A}\subset C_{\bar A}\cup \mathbb{S}^{1} 
$$
keeping in mind that the interiors and closures are with respect to different spaces,
in one case \(\mathbb{R}^2 \) and in one case \(\mathbb{S}^{1} \).
Therefore the above computation shows that 
$$
L(C_A^o,2,\alpha)=
\frac{ C_\alpha^2  }{8}
\int_{A^o} \alpha (\cos \theta \sin \theta)^{-(1 + \alpha)} \, d\theta 
$$
and 
$$
L(\overline{C_A},2,\alpha)=
\frac{ C_\alpha^2  }{8}
\int_{\bar A} \alpha (\cos \theta \sin \theta)^{-(1 + \alpha)} \, d\theta 
$$
where for the latter equation, we also used the fact that 
the \(\mathbb{S}^{1}\) piece adds nothing to the relevant integral.

Now, using the fact the boundary of $A$ has measure zero, we conclude that
$L((C_A)^o,2,\alpha)=L(\overline{C_A},2,\alpha)$. Since $\varepsilon$ is fixed, it is easy
to see that $L((C_A)_{\delta,+},2,\alpha)$ is finite for sufficiently small
$\delta$ allowing us to conclude that
$L((C_A)^o,2,\alpha)=\lim_{\delta \to 0}L((C_A)_{\delta,+},2,\alpha)$.
Theorem~\ref{theorem: later tails} with \( k = 2 \) now yields the result.
\end{proof}

\begin{remark}
This improves on~\eqref{eq: cone result} in this case since it yields the correct decay
rate and demonstrates the hidden regular variation behavior.
The former result would only give
$\lim_{h \to \infty} h^{\alpha} P(X \in C_A)=0$.
Not surprisingly, when $A$ is as large as possible with $\varepsilon$ fixed,
the integral tends to infinity as $\varepsilon$ goes to 0; this is because we are getting
closer to the support of the spectral measure.
\end{remark}

\begin{figure}[ht]

\centering
\begin{tikzpicture}[scale=0.8]
\begin{axis}[axis x line=middle,axis y line=middle, xmin=-3,xmax=3, ymin=-3,ymax=3, xlabel=$\mathbf{x}_1$, ylabel=$\mathbf{x}_2$, x label style={at={(axis description cs:1.05, 0.54)},anchor=north}, y label style={at={(axis description cs:0.5, 1.08)},anchor=north}, xtick={0}, ytick={0}]

\draw[dashed,domain=35:70] plot (axis cs: {cos(\x)}, {sin(\x)});
\draw[dashed,domain=1:4] plot (axis cs: {0.82* \x}, {0.57* \x}); 
\draw[dashed,domain=1:4] plot (axis cs: {0.34* \x}, {0.94* \x}); 
   
\fill[fill = gray!30, opacity=.4] (axis cs: 3.28, 2.29) --   (axis cs: 0.82, 0.57)  --   (axis cs: 0.7, 0.74)  -- (axis cs: 0.34, 0.94)  -- (axis cs: 1.37, 3.76)  -- (axis cs: 3.28, 3.76) ;

\draw[green,domain=0:360] plot (axis cs: {0.5*cos(\x)}, {0.5*sin(\x)});

\fill[red] (axis cs: 0.5,0) circle (0.6mm); 
\fill[red] (axis cs: -0.5,0) circle (0.6mm); 
\fill[red] (axis cs: 0,0.5) circle (0.6mm); 
\fill[red] (axis cs: 0,-0.5) circle (0.6mm); 
\end{axis}
\end{tikzpicture} 

\caption{The figure above shows the set \( 2C_A \) (gray) considered in Case (ii)
of Example~\ref{example: independent marginals}, together with the four 
points (in red) at which \( \Lambda \) has support.}
\end{figure}
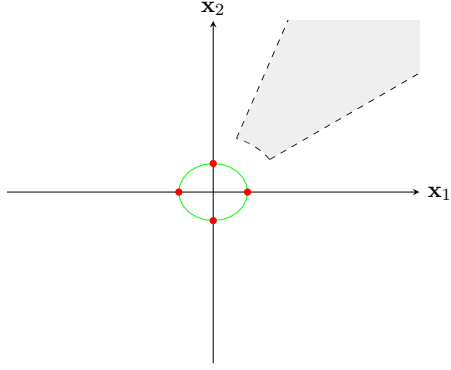

\paragraph{Case (iii)} This example, while fairly simple, has three different values arising in~\eqref{eq:theorem1} when $k=1$ and, in particular, Theorem~\ref{theorem: later tails} yields nonmatching upper and lower bounds. 
We let 
\[ 
E = \{ \mathbf{x} \in \mathbb{R}^2 \colon {x}_1 >1, {x}_2<0 \}.
\]
It is easy to check that for any \( \alpha \in (0,2) \), we have that $L(E^o,1, \alpha)=0$, $L(\bar E,1, \alpha)=\lim_{\delta \to 0} L(E_{\delta,+},1, \alpha)=C_\alpha/2$ while using the independence of the components, it is immediate that the middle terms in~\eqref{eq:theorem1} when $k=1$ are $C_\alpha/4$. 


%
\begin{figure}[ht]

\centering
\begin{tikzpicture}[scale=0.8]
\begin{axis}[axis x line=middle,axis y line=middle, xmin=-3,xmax=3, ymin=-3,ymax=3, xlabel=$\mathbf{x}_1$, ylabel=$\mathbf{x}_2$, x label style={at={(axis description cs:1.05, 0.54)},anchor=north}, y label style={at={(axis description cs:0.5, 1.08)},anchor=north}, xtick={0}, ytick={0}]
\fill[fill = gray!30, opacity=.4] (axis cs: 1,-4) rectangle (axis cs: 2.9,0);
\draw[dashed] (axis cs: 1,-3) -- (axis cs: 1,0) -- (axis cs: 2.9,0);
\draw[red] (axis cs: -2.9,-0) -- (axis cs: 2.9,0);

\draw[green,domain=0:360] plot (axis cs: {0.5*cos(\x)}, {0.5*sin(\x)});

\fill[red] (axis cs: 0.5,0) circle (0.6mm); 
\fill[red] (axis cs: -0.5,0) circle (0.6mm); 
\fill[red] (axis cs: 0,0.5) circle (0.6mm); 
\fill[red] (axis cs: 0,-0.5) circle (0.6mm); 
\end{axis}
\end{tikzpicture} 

\caption{The figure above shows the set \( 2E \) (gray) considered in Case (iii)
of Example~\ref{example: independent marginals} together with the four points (in red) at which 
\( \Lambda \) has support.}
\end{figure}
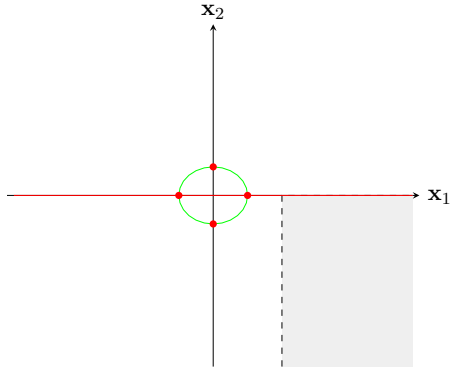


\begin{figure}[ht]

\centering
\begin{tikzpicture}[scale=0.8]
\begin{axis}[axis x line=middle,axis y line=middle, xmin=-3,xmax=3, ymin=-3,ymax=3, xlabel=$\mathbf{x}_1$, ylabel=$\mathbf{x}_2$, x label style={at={(axis description cs:1.05, 0.54)},anchor=north}, y label style={at={(axis description cs:0.5, 1.08)},anchor=north}, xtick={0}, ytick={0}]

\fill[fill = gray!30, opacity=.4] (axis cs: 1,1) rectangle (axis cs: 2.9,-3);
\draw[dashed] (axis cs: 1,-3) -- (axis cs: 1,1) -- (axis cs: 2.9,1);

\draw[red] (axis cs: 0,-0) -- (axis cs: 2.9,2.9);

\draw[green,domain=0:360] plot (axis cs: {0.5*cos(\x)}, {0.5*sin(\x)});

\fill[red] (axis cs: 0.36,0.36) circle (0.6mm); 
\fill[red] (axis cs: -0.36,-0.36) circle (0.6mm); 
\fill[red] (axis cs: 0,-0.5) circle (0.6mm); 
\fill[red] (axis cs: 0,0.5) circle (0.6mm); 
\end{axis}
\end{tikzpicture} 

\caption{The figure above shows the set \( 2E \) (gray) considered 
in Example~\ref{example: from paper} together with the four points (in red) at 
which \( \Lambda \) has support.}
\end{figure}
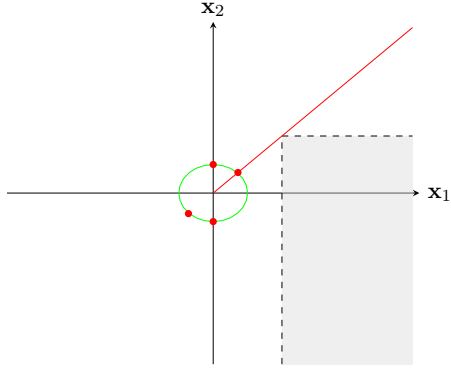

\end{example}

Our next example illustrates a number of interesting phenomena which we summarize in  
Proposition~\ref{prop: cool.example2}
after giving the example. This provides an example where (i) the decay rate has 
three possible behaviors
depending on $\alpha$, (ii) $L(E^o,k,\alpha)\neq \lim_{\delta \to 0} L(E_{\delta,+},k,\alpha)$
and (iii) where the tail behavior can drastically change due to a modification in the set $E$ 
in an arbitrarily small neighborhood of one point, namely $(1,1)$. 
It is also a ``baby version'' of the example following it which will be crucially
used in \cite{PS.threshold}. 

\begin{example}\label{example: from paper}
Let \( \alpha \in (0,2) \) and \( S_1 \) and \( S_2 \) be i.i.d.\ with distribution {\( S_\alpha \)} and let
\[
X = (1,1) S_1 - (0,1)S_2.
\]
Then \( X \) is a symmetric \( \alpha \)-stable random vector and
the spectral measure \( \Lambda \) of 
\( X \)  has mass \( 2^{\alpha/2}/2\) at \( \pm (1,1) /\sqrt{2}\) and mass \( 1/2 \) at \( \pm(0,1) \). 
Let
\[
E \coloneqq \bigl\{ \mathbf{x} \in \mathbb{R}^2 \colon {x}_1 >1, {x}_2<1 \bigr\}.
\]
We mention that it is straightforward to show
that for all $\alpha$, \(\lim_{\delta \to 0} L(E_{\delta,+},1,\alpha)=0\). 

\begin{proposition} \label{prop: cool.example2}
Let $\Lambda$, $X$ and $E$ be as above. 

\begin{enumerate}[(i)]
\item 
\begin{itemize}
\item
For \(\alpha < 1\), 
\begin{equation}\label{eq: exactlimit.small.alpha}
\lim_{h\to \infty} h^{2 \alpha} P(X \in h E )=
 \frac{C_\alpha^2 \alpha \Gamma(2\alpha) \Gamma(1-\alpha)}{4\Gamma(1+\alpha)}<\infty.
\end{equation}

\item
For \(\alpha = 1\), 
\begin{equation}\label{eq: critalphaasymp}
\lim_{h\to \infty} \frac{h^2}{\log h} P(X \in h E )= \frac{C_1^2}{4}.
\end{equation}

\item
For \(\alpha > 1\), 
\begin{equation}\label{eq: largealphaasymp}
\lim_{h\to \infty} h^{1+ \alpha} P(X \in h E )= \frac{C_\alpha\alpha\E[|S_1|]}{4}.
\end{equation}

\end{itemize}

\item For all $\alpha \in (0,2)$, $\lim_{\delta \to 0} L(E_{\delta,+},2,\alpha)=\infty$. Moreover,
$L(E^o,2,\alpha)=L(\bar E,2,\alpha)$ is equal to $\infty$ if $\alpha \in [1,2)$ and is
equal to \(\frac{C_\alpha^2 \alpha \Gamma(2\alpha) \Gamma(1-\alpha)}{4\Gamma(1+\alpha)}\)
if $\alpha \in (0,1)$.

\item Let $B_\varepsilon\coloneqq B_\infty((1,1),\varepsilon)$ be
the ball around $(1,1)$ of radius $\varepsilon$ in the $L_\infty$ metric. 
For any \(\varepsilon>0\) and $\alpha\in (0,1)$,
$$
g^+(\alpha,\varepsilon)\coloneqq L(E^o\cup B_\varepsilon,1,\alpha)=
\lim_{\delta \to 0} L((E\cup B_\varepsilon)_{\delta,+},1,\alpha)\in (0,\infty)
$$
implying by Theorem~\ref{theorem: later tails} that  
$$
\lim_{h\to \infty} h^{\alpha} P(X \in h (E\cup B_\varepsilon))=g^+(\alpha,\varepsilon)
$$
and 
$$
g^-(\alpha,\varepsilon)\coloneqq L((E\backslash B_\varepsilon)^o,2,\alpha)=
\lim_{\delta \to 0} L((E\backslash B_\varepsilon)_{\delta,+},2,\alpha)\in (0,\infty)
$$
implying by Theorem~\ref{theorem: later tails} that  
$$
\lim_{h\to \infty} h^{2\alpha} P(X \in h (E\backslash B_\varepsilon))=g^-(\alpha,\varepsilon).
$$

\item For all $\alpha\in (0,2)$,
$g^+(\alpha,\varepsilon)$ is increasing in $\varepsilon$ with
$\lim_{\varepsilon \to 0}g^+(\alpha,\varepsilon)=0$ while
$g^-(\alpha,\varepsilon)$ is decreasing in $\varepsilon$ with
$\lim_{\varepsilon \to 0}g^-(\alpha,\varepsilon)=\infty$ for 
$\alpha\in [1,2)$ and 
$\frac{C_\alpha^2 \alpha \Gamma(2\alpha) \Gamma(1-\alpha)}{4\Gamma(1+\alpha)}$
for $\alpha\in (0,1)$.
\end{enumerate}
\end{proposition}

\begin{proof}
We only prove (i) and (ii). (iii) and (iv) are fairly straightforward and left to the reader.
We start with the proof of (ii).

It is easy to see that $L(E^o,2,\alpha)=L(\bar E,2,\alpha)$ and that their common value is
\begin{align*}
& C_\alpha^2  \cdot 2^{\alpha/2}/2 \cdot 1/2 \cdot 
\int_{\sqrt{2}}^\infty
\int_{(s_1/\sqrt{2})-1}^\infty  \alpha^2  s_1^{-(1+\alpha)}  s_2^{-(1+\alpha)}\, ds_2 \, ds_1
\\&\qquad  = 
\frac{C_\alpha^2 }{4}
\int_{1}^\infty
\int_{t_1-1}^\infty  \alpha^2  t_1^{-(1+\alpha)}  s_2^{-(1+\alpha)}\, ds_2 \, dt_1   
\\&\qquad =
\frac{C_\alpha^2 }{4}
\int_1^\infty
 \alpha  t_1^{-(1+\alpha)}  \left[ -s_2^{-\alpha} \right]_{t_1-1}^\infty  \, dt_1
 \\&\qquad =
\frac{C_\alpha^2 }{4}
\int_1^\infty
   \alpha  t_1^{-(1+\alpha)}(t_1-1)^{-\alpha}   \, dt_1
   \\&\qquad = 
\frac{C_\alpha^2 }{4}
\int_0^1
   \alpha  x^{2\alpha-1}(1-x)^{-\alpha}   \, dx.
\end{align*}
 This integral is easily verified to be infinite if and only if 
\( \alpha \geq 1 \) and  strictly positive for all \( \alpha \in (0,1) \).
Recognizing the integrand 
as the probability density function (up to a constant) of a
Beta distribution with parameters $2\alpha$ and $1-\alpha$, we see that the last expression is equal to 
\[
\frac{C_\alpha^2 \alpha \Gamma(2\alpha) \Gamma(1-\alpha)}{4\Gamma(1+\alpha)}.
\] 
The fact that $\lim_{\delta \to 0} L(E_{\delta,+},2,\alpha)$ is $\infty$ is seen by noting that for any fixed $\delta>0$, the term
$$
\Lambda^2 \left( \mathbf{x}_1, \mathbf{x}_2 \in \mathbb{S}^{1} \colon   
s_1\mathbf{x}_1 + s_2\mathbf{x}_2 \in E_{\delta,+} \right)
$$
is uniformly bounded away from 0 for arbitrarily small $s_2$ and hence the integral diverges.
This finishes the proof of (ii).

We now move to (i). 
Since \( X = (1,1) S_1 - (0,1)S_2 \), we have 
\[ 
\{X \in hE \} = \{h < S_1 < h + S_2\}
\]
and so for any $\alpha$, we have
$$
P(X \in hE) =  \int_0^{\infty} f(t) P(h < S_1 < h+t)\, dt .
$$

We now proceed with the \( \alpha \in (0,1) \) case. It is not hard to show that for every $\varepsilon >0$,
$$
L((E\backslash  B_\varepsilon)^o,2,\alpha)=
\lim_{\delta \to 0} L((E\backslash  B_\varepsilon)_{\delta,+},2,\alpha)\in (0,\infty)
$$
and hence by Theorem~\ref{theorem: later tails} 
\[ 
\lim_{h\to \infty}h^{2\alpha}P(X \in h (E \backslash B_\varepsilon)) =
L((E\backslash  B_\varepsilon)^o,2,\alpha).
\]
Letting \(\varepsilon\to 0\), we can apply the monotone convergence theorem to both sides
(using the fact that $E$ is open) and conclude~\eqref{eq: exactlimit.small.alpha} as desired.

Now instead let \( \alpha =1 \). It is not hard to show that for every $\varepsilon >0$,
by breaking up the following integral into $[0,h]$ and $[h,\infty)$ and using the fact that $f$ is 
decreasing, we have
$$
\frac{h^2}{\log h} P(X \in h E )=  
\frac{h^2}{\log h}  \int_0^{\infty} f(t) P(h < S_1 < h+t)\, dt 
$$
$$
\le \frac{h^2}{\log h} \Bigl[   f(h)  \int_0^{h} f(t) t\, dt + P(S_1\ge h)^2\Bigr].
$$
Noting that~\eqref{eq: basictail} implies the second term goes to 0 as $h\to \infty$ and
the fact that that~\eqref{eq: densitytail} easily implies that
\begin{equation}\label{eq: cesaro}
\lim_{h\to \infty} \frac{\int_0^{h} f(t) t\, dt}{\log h}=C_1/2
\end{equation}
as well as applying~\eqref{eq: densitytail} directly, we get
$$
\limsup_{h\to\infty}\frac{h^2}{\log h} P(X \in h E )\le
\frac{C_1^2}{4}.
$$
For the lower bound, we fix $\varepsilon >0$, integrate only over $[0,\varepsilon h]$ and use $f$ is decreasing
to obtain
$$
\frac{h^2}{\log h} P(X \in h E )\ge 
\frac{h^2}{\log h}  f((1+\varepsilon)h)  \int_0^{\varepsilon h} f(t) t\, dt.
$$
Using~\eqref{eq: densitytail} and~\eqref{eq: cesaro}, the limit of the last term is, as $h\to\infty$,
equal to
$C_1^2/4(1+\varepsilon)^2$. Hence for every $\varepsilon >0$, we have
$$
\liminf_{h\to\infty}\frac{h^2}{\log h} P(X \in h E )\ge\frac{C_1^2}{4(1+\varepsilon)^2}
$$
and we can then let $\varepsilon\to 0$ to complete the proof.

Finally, we now do the case \( \alpha \in (1,2) \).
Using the fact that $f$ is decreasing and using~\eqref{eq: densitytail}, we have
\begin{equation}\label{eq: upperbound}
P(X \in hE) =  \int_0^{\infty} f(t) P(h < S_1 < h+t)\, dt 
\le f(h) \int_0^{\infty} f(t)t \, dt 
\end{equation}
$$
\sim \frac{C_\alpha\alpha\E[|S_1|]}{4}h^{-(1+\alpha)}
$$
establishing the upper bound in~\eqref{eq: largealphaasymp}. 
For the lower bound,  fixing $\varepsilon>0$, we have
\begin{equation}\label{eq: lowerbound}
\begin{split}
&P(X \in hE) =  \int_0^{\infty} f(t) P(h < S_1 < h+t)\, dt 
\ge \int_0^{\varepsilon h} f(t) P(h < S_1 < h+t)\, dt 
\\&\qquad \ge
f(h(1+\varepsilon)) \int_0^{\varepsilon h} f(t)t \, dt 
\sim (1+\varepsilon)^{-(1+\alpha)} \, \frac{C_\alpha\alpha\E[|S_1|]}{4}h^{-(1+\alpha)}.
\end{split}
\end{equation}
It follows that 
$$
\liminf_{h\to\infty} h^{(1+\alpha)}P(X \in hE) \ge
(1+\varepsilon)^{-(1+\alpha)} \, \frac{C_\alpha\alpha\E[|S_1|]}{4}.
$$
One can now let $\varepsilon\to 0$, obtaining the lower bound in~\eqref{eq: largealphaasymp}, completing
the proof.
\end{proof}

\begin{remark}
If $X$ is as in our first example where we have independent components, 
one can construct a set, namely
\[
E \coloneqq \{ \mathbf{x} \in \mathbb{R}^2 \colon {x}_1 >1, {x}_2>a(x-1) \}
\]
for $a \in (0,1)$, which exhibits similar behavior to that in the above proposition. However, the above 
example, when generalized to three variables, is what we need in another context and so we proceeded in 
this way.
\end{remark}

\begin{remark}\label{remark: other phase transition point}
With the previous result in mind, one might wonder if any threshold  for events of the type \( \{ X \in hE \} \) will 
occur at \( \alpha = 1 \). To show that this is not the case, fix \( \alpha \in (0,2) \)  and \( \sigma >0 \), and define
\[
E_\sigma \coloneqq \bigl\{ \mathbf{x} \in \mathbb{R}^2 \colon 1 < {x}_1 < 1 + {x}_2^\sigma \bigr\}.
\]
Further,  let \( S_1, S_2 \sim S_\alpha \) be i.i.d\ and consider the decay rate of   \( P(X \in hE_\sigma) \) 
as \( h \to \infty \).
Then, using a very similar argument to the argument in the proof of Proposition~\ref{prop: cool.example2},
one can show that we  get a phase transition in the behavior of the decay rate of \( P(X \in hE_\sigma) \) at 
\( \alpha = \sigma \), and in fact
\[
P\bigl(X \in hE_\sigma  \bigr) \asymp \begin{cases} h^{-(\alpha + \sigma)} &\text{if } \alpha > \sigma \cr h^{-2 \alpha} \log h & \text{if } \alpha = \sigma \cr h^{-2\alpha} &\text{if } \alpha < \sigma.\end{cases}
\]

\begin{figure}[ht]

\centering
\subfigure[\( \sigma = 0.5 \)]{
\centering
\begin{tikzpicture}[scale=0.75]
\begin{axis}[axis x line=middle,axis y line=middle, xmin=-3,xmax=3, ymin=-3,ymax=3, xlabel=$\mathbf{x}_1$, ylabel=$\mathbf{x}_2$, x label style={at={(axis description cs:1.05, 0.54)},anchor=north}, y label style={at={(axis description cs:0.5, 1.08)},anchor=north}, xtick={0}, ytick={0}]
 
\draw[dashed,name path=A] (axis cs: 1,0) -- (axis cs: 1,4.2);
\addplot[dashed, name path=B, no markers, samples=200,domain=1:4] {(x-1)^(1/0.5)};
 
\tikzfillbetween[of=A and B]{fill = gray!30, opacity=.4};
 
\draw[green,domain=0:360] plot (axis cs: {0.5*cos(\x)}, {0.5*sin(\x)});

\fill[red] (axis cs: 0.5,0) circle (0.6mm); 
\fill[red] (axis cs: -0.5,0) circle (0.6mm); 
\fill[red] (axis cs: 0,-0.5) circle (0.6mm); 
\fill[red] (axis cs: 0,0.5) circle (0.6mm); 
\end{axis}
\end{tikzpicture} 
}
\subfigure[\( \sigma = 1.8 \)]{
\centering
\begin{tikzpicture}[scale=0.75]
\begin{axis}[axis x line=middle,axis y line=middle, xmin=-3,xmax=3, ymin=-3,ymax=3, xlabel=$\mathbf{x}_1$, ylabel=$\mathbf{x}_2$, x label style={at={(axis description cs:1.05, 0.54)},anchor=north}, y label style={at={(axis description cs:0.5, 1.08)},anchor=north}, xtick={0}, ytick={0}]
 
\draw[dashed,name path=A] (axis cs: 1,0) -- (axis cs: 1,4.2);
 \addplot[dashed, no markers, samples=200,name path=B] {(x-1)^(1/1.8)};
 
\tikzfillbetween[of=A and B]{fill = gray!30, opacity=.4};
 
\draw[green,domain=0:360] plot (axis cs: {0.5*cos(\x)}, {0.5*sin(\x)});

\fill[red] (axis cs: 0.5,0) circle (0.6mm); 
\fill[red] (axis cs: -0.5,0) circle (0.6mm); 
\fill[red] (axis cs: 0,-0.5) circle (0.6mm); 
\fill[red] (axis cs: 0,0.5) circle (0.6mm); 
\end{axis}
\end{tikzpicture} 
} 

\caption{The figures above shows the set \( 2E_\sigma \) in Remark~\ref{remark: other phase transition point}, for two different values of \( \sigma \), together with the four points (in red) at which the spectral measure \( \Lambda \) from the same remark has support.}\label{fig: last fig}
\end{figure}
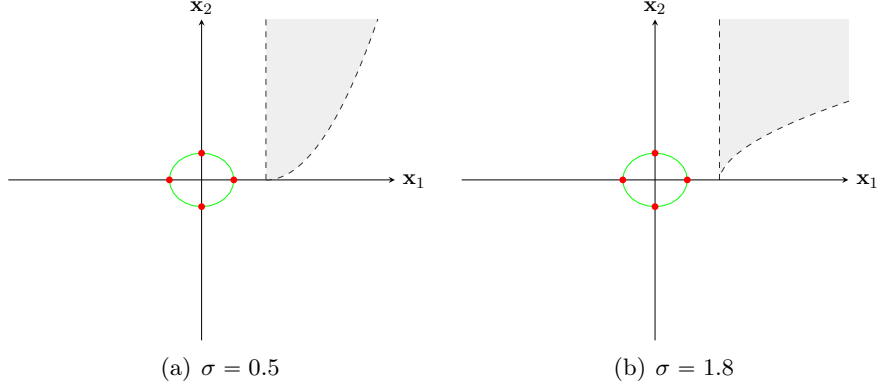

\end{remark}
\end{example}

In our next, and final, example we study one of the simplest three-dimensional permutation invariant multivariate stable distributions, and show that it exhibits the same behavior as our previous example. Here we   only study the case \( \alpha \in (0,1) \) in detail, but the cases \( \alpha = 1 \) and \( \alpha > 1 \) can be done similarly as in the the proof of Proposition~\ref{prop: cool.example2}. 
\begin{example}\label{example: from paper II}
Let \( \alpha \in (0,2) \) and let \( S_0 \), \( S_1 \), \( S_2 \) and \( S_3 \) be i.i.d.\ with \( S_0 \sim S_\alpha \). Furthermore, let \( a \in (0,1) \) and define \( X_1 \), \( X_2 \) and \( X_3 \) by
\[
 X_{i } \coloneqq aS_{0} + (1-a^\alpha)^{1/\alpha} S_i , \qquad i=1,2,3.
\]
Clearly \( (X_1,X_2,X_3) \) is a three-dimensional symmetric \( \alpha \)-stable random vector 
whose marginals are \( S_\alpha \). The corresponding spectral measure  \( \Lambda \)  has mass \( a^\alpha 3^{\alpha/2} /2 \) at \( \pm (1,1,1)/\sqrt{3} \) and mass \( (1-a^\alpha)/2 \) at \( \pm(1,0,0)\), \( \pm (0,1,0 ) \) and \( \pm (0,0,1) \). 
Consider the set 
\[
E \coloneqq \{ \mathbf{x} \in \mathbb{R}^3 \colon {x}_1,\,  {x}_2>1 , \mathbf{x}_3<1\}.
\]

The proof of the  following proposition follows the proof of Proposition~\ref{prop: cool.example2} exactly, and therefore we only give a sketch of the proof here.
\begin{proposition}
Let $\Lambda$, \( X \) and $E$ be as above. 
Then for all \( \alpha \in (0,1)\),  we have that
\begin{equation}\label{eq: exactlimit.small.alpha II}
\lim_{h\to \infty} h^{2 \alpha} P(X \in h E )= 
 \frac{C_\alpha^2}{4}
\Biggl(
 (1-a^\alpha)^2 
+
 a^\alpha(1-a^\alpha)
\cdot \frac{ \alpha \Gamma(2\alpha) \Gamma(1-\alpha)}{\Gamma(1+\alpha)}
   \Biggr)<\infty.
\end{equation}
Moreover, for all $\alpha \in (0,2)$, $\lim_{\delta \to 0} L(E_{\delta,+},2,\alpha)=\infty$, and
$L(E^o,2,\alpha)=L(\bar E,2,\alpha)$ is equal to $\infty$ if $\alpha \in [1,2)$ and is
equal to the right hand side of~\eqref{eq: exactlimit.small.alpha II}
if $\alpha \in (0,1)$.  

\end{proposition}

\begin{proof}[Proof sketch]
It is easy to see that \( L(E^o,2,\alpha) = L(\bar E,2,\alpha) \) and that their common value is
\begin{align*}
&
\Lambda(e_1) \Lambda(e_2) \cdot C_\alpha^2  
\int_1^\infty
\int_{1}^\infty  \alpha^2  s_1^{-(1+\alpha)}  s_2^{-(1+\alpha)}\, ds_1 \, ds_2
\\&\qquad\qquad +
\Lambda((1,1,1)/\sqrt{3}) \Lambda(-e_3)
\cdot
C_\alpha^2  
\int_{\sqrt{3}}^\infty
\int_{(s_0/\sqrt{3})-1}^\infty  \alpha^2  s_0^{-(1+\alpha)}  s_3^{-(1+\alpha)}\, ds_3 \, ds_0
\\&\qquad  =
\frac{C_\alpha^2}{4}
\left(
 (1-a^\alpha)^2 
+
 a^\alpha(1-a^\alpha)
\int_0^1
   \alpha  x^{2\alpha-1}(1-x)^{-\alpha}   \, dx
   \right).
\end{align*}
The rest of the proof follows the lines of 
the proof of Proposition~\ref{prop: cool.example2} exactly, and is hence omitted here.

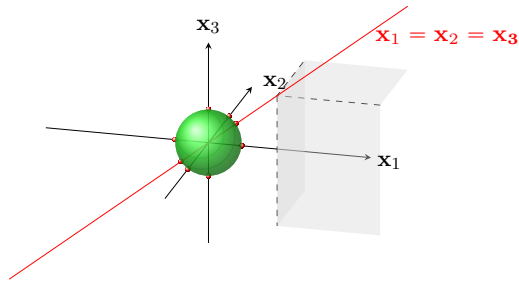
\begin{figure}[ht]

\centering

\begin{tikzpicture}[scale=0.8]

\draw[white] (3.42-6,2) -- (3.42+6,2);

\begin{axis}[axis lines=middle, xmin=-3,xmax=3, ymin=-3,ymax=3, zmin=-3,zmax=3, xtick={0}, ytick={0}, ztick = {0},
xlabel=$\mathbf{x}_1$, 
ylabel=$\mathbf{x}_2$,
zlabel=$\mathbf{x}_3$,
 every axis x label/.style={
  at={(axis cs:\pgfkeysvalueof{/pgfplots/xmax},0,0)},
  xshift=0.8em,yshift=-0.2em                 
  },
  every axis y label/.style={
  at={(axis cs:0,\pgfkeysvalueof{/pgfplots/ymax},0)},
  xshift=1em,yshift=0.2em            
  },
  every axis z label/.style={
  at={(axis cs:0,0,\pgfkeysvalueof{/pgfplots/zmax})},
  xshift=0em,yshift=0.7em           
  },
  view={15}{30} ]

 \draw[dashed] (axis cs: 2.9,1,1) -- (axis cs: 1,1,1) -- (axis cs: 1,2.9,1);
 \draw[dashed]   (axis cs: 1,1,1) -- (axis cs: 1,1,-2.9);
 
 \fill [fill = gray!30, opacity=.4] (axis cs: 2.9,1,1) -- (axis cs: 1,1,1) -- (axis cs: 1,2.9,1) -- (axis cs: 2.9,2.9,1) ;
 \fill [fill = gray!30, opacity=.4] (axis cs: 2.9,1,1) -- (axis cs: 1,1,1) -- (axis cs: 1,1,-2.9) -- (axis cs: 2.9,1,-2.9) ;
 \fill [fill = gray!30, opacity=.4] (axis cs: 1,2.9,1) -- (axis cs: 1,1,1) -- (axis cs: 1,1,-2.9) -- (axis cs: 1,2.9,-2.9) ;
 
 \shade[ball color=red] (axis cs: 0.62,0,0) circle (.4mm);
 \shade[ball color=red] (axis cs: -0.62,0,0) circle (.4mm);
 
  \shade[ball color=red] (axis cs: 0,1.42,0) circle (.4mm);
 \shade[ball color=red] (axis cs: 0,-1.42,0) circle (.4mm);
 
  \shade[ball color=red] (axis cs: 0,0,1) circle (.4mm);
 \shade[ball color=red] (axis cs: 0,0,-1) circle (.4mm);

  \shade[ball color=red] (axis cs: 0.4,0.4,0.4) circle (.4mm);
  \shade[ball color=red] (axis cs: -0.4,-0.4,-0.4) circle (.4mm);
  
\draw[red] (axis cs: -2.9,-2.9,-2.9) -- (axis cs: 2.9,2.9,2.9) ;

 \shade[ball color=green!80!white, opacity=0.5] (axis cs: 0,0,0) circle (.55cm);
 
\end{axis}

\draw[red] (7.4,4.6) node {\footnotesize $\mathbf{x}_1 = \mathbf{x}_2 = \mathbf{x_3}$};

\end{tikzpicture} 

\caption{The figure above shows the set \( 2E \) (gray) considered in Example~\ref{example: from paper II} together with the eight points (in red) at which \( \Lambda \) has support.}
\end{figure}

\end{proof}
\end{example}

\noindent\textbf{Acknowledgements.}
We thank Gennady Samorodnitsky for private correspondence and for elaborating on one 
of the proofs in~\cite{st1994}. We also  thank Thomas Mikosch for informing 
us about the notion of hidden regular variation and its relationship to our work and two anonymous referees for various useful comments. 

The first author acknowledges support from the European Research Council,  grant no.\ 682537. The second author acknowledges the support of the Swedish Research Council, grant no.\ 2016-03835 and the Knut and Alice Wallenberg Foundation, grant no.\ 2012.0067.

\end{document}